\documentclass[10pt]{article}
\usepackage{amsthm}
\usepackage{amsmath}
\usepackage{amssymb}
\usepackage{makeidx}

\theoremstyle{definition}
\newtheorem{definition}{Definition}[section]

\newtheorem{openproblem}[definition]{Open problem}

\newtheorem{remark}[definition]{Remark}

\theoremstyle{plain}
\newtheorem{theorem}[definition]{Theorem}
\newtheorem{lemma}[definition]{Lemma}
\newtheorem{proposition}[definition]{Proposition}
\newtheorem{corollary}[definition]{Corollary}

\numberwithin{equation}{section}

\def\N{{\mathbb N}}

\begin{document}
\title{Transferring Davey`s Theorem on Annihilators in Bounded\\ Distributive Lattices to Modular Congruence Lattices and Rings}
\author{Claudia MURE\c SAN\thanks{Dedicated to the memory of my dear grandmother, Floar\u a--Marioara Mure\c san}\\ \footnotesize University of Bucharest\\ \footnotesize Faculty of Mathematics and Computer Science\\ \footnotesize Academiei 14, RO 010014, Bucharest, Romania\\ \footnotesize E--mails: cmuresan@fmi.unibuc.ro, c.muresan@yahoo.com}
\date{\today }
\maketitle

\begin{abstract} Congruence lattices of semiprime algebras from semi--degenerate congruence--modular varieties fulfill the equivalences from B. A. Davey`s well--known characterization theorem for $m$--Stone bounded distributive lattices; moreover, changing the cardinalities in those equivalent conditions does not change their validity. I prove this by transferring Davey`s Theorem from bounded distributive lattices to such congruence lattices through a certain lattice morphism and using the fact that the codomain of that morphism is a frame. Furthermore, these equivalent conditions are preserved by finite direct products of such algebras, and similar equivalences are fulfilled by the elements of semiprime commutative unitary rings and, dualized, by the elements of complete residuated lattices.\end{abstract}

\section{Introduction}
\label{introduction}

In \cite{eu,eu7}, I have transferred \cite[Theorem $1$]{dav} from bounded distributive lattices to residuated lattices, by using the reticulation of a residuated lattice. In \cite{retic}, G. Georgescu and I have constructed the reticulation for algebras from congruence--modular varieties. I have noticed that the kind of transfer from \cite{eu,eu7} can be made in this general context, but referring to the lattices of congruences of the algebras from such varieties instead of their elements and enforcing some restrictions: the varieties must be semi--degenerate and the algebras in question must be semiprime. It turns out that the transfer of these properties doesn`t even necessitate the reticulation, but only part of its construction from \cite{retic}, and it produces further equivalences, because changing the cardinality in those conditions forms other properties which are equivalent to those conditions. The present work contains two main results, the first of which is Theorem \ref{mydavey} below, stating the equivalences fulfilled by the congruence lattices of such algebras, which I am proving along with some related results, such as the fact that, if the congruence lattice of such an algebra satisfies the equivalent conditions from Theorem \ref{mydavey}, then the reticulation of that algebra satisfies those conditions, as well, and the fact that those conditions are preserved by finite direct products of such algebras. While the structure of the proof of Theorem \ref{mydavey} resembles the one I have made for residuated lattices in \cite{eu,eu7}, the auxiliary results for that proof do not hold in the present context, which, in turn, produces new auxiliary results, so all the following results are new and original, excepting only the ones cited from other works or mentioned as being either well known or immediate from well--known properties.

The natural question that arises is whether other algebras fulfill analogues of Theorem \ref{mydavey} or \cite[Theorem $1$]{dav} for elements instead of congruences, as is the case for residuated lattices. I prove that semiprime commutative unitary rings do, in the second main result of the present paper: Theorem \ref{myringsdavey}. It also turns out that complete residuated lattices fulfill the dual of Theorem \ref{mydavey} expressed for elements, because, for such residuated lattices, changing the cardinalities in the equivalent conditions from the analogue of Davey`s Theorem for residuated lattices from \cite{eu,eu7} produces other properties equivalent to those conditions. Of course, both residuated lattices, which are congruence--distributive, thus semiprime, and form a semi--degenerate variety, and semiprime commutative unitary rings, which are semiprime algebras from the semi--degenerate congruence--modular variety of commutative unitary rings, fulfill those e\-qui\-va\-len\-ces for congruences, too.

Section \ref{resultscg} contains Theorem \ref{mydavey} and the related results on congruence lattices. Section \ref{resultselem} contains a brief presentation of the situation in residuated lattices, the analogue of Theorem \ref{mydavey} for the elements of semiprime commutative unitary rings instead of the congruences from congruence lattices of semiprime algebras from semi--degenerate congruence--modular varieties, and some related results. Section \ref{conclusions} contains a brief layout of some directions for future research.

\section{Transferring Davey`s Theorem to Congruence Lattices}
\label{resultscg}

For any set $S$, $|S|$ shall denote the cardinality of $S$ and by $|S|<\infty $ we shall specify the fact that $S$ is finite. And, throughout this paper, $m$ shall be an infinite cardinal, arbitrary but fixed. For brevity, instead of treating separate cases, we shall often use the fact that, in any bounded poset, $\bigvee \emptyset $ is the minimum and $\bigwedge \emptyset $ is the maximum of that poset. Throughout this paper, we shall designate any algebra by its underlying set, unless there is danger of confusion.

Let $(L,\vee ,\wedge ,0,1)$ be an arbitrary bounded lattice. $({\rm Id}(L),\vee ,\cap ,\{0\},L)$ shall be the complete lattice of the ideals of $L$, which is well known to be distributive exactly when $L$ is distributive. Let $a\in L$ and $U,V\subseteq L$. We shall denote by $(U]$ the ideal of $L$ generated by $U$, by $(a]=(\{a\}]$ the principal ideal of $L$ generated by $a$, by ${\rm Ann}(U)$ the {\em annihilator of $U$} and by ${\rm Ann}(a)$ the {\em annihilator of $a$} in L: ${\rm Ann}(U)=\{x\in L\ |\ (\forall \, u\in U)\, (x\wedge u=0)\}$ and ${\rm Ann}(a)={\rm Ann}(\{a\})=\{x\in L\ |\ x\wedge a=0\}$. Whenever specifying $L$ is necessary, we shall denote $(U]_L$, $(a]_L$, ${\rm Ann}_L(U)$ and ${\rm Ann}_L(a)$ instead of $(U]$, $(a]$, ${\rm Ann}(U)$ and ${\rm Ann}(a)$, respectively. Obviously, ${\rm Ann}(a)={\rm Ann}((a])$, $U\subseteq {\rm Ann}({\rm Ann}(U))$ and $\displaystyle {\rm Ann}(U)=\bigcap _{u\in U}{\rm Ann}(u)$, hence $\displaystyle {\rm Ann}(U)\cap {\rm Ann}(V)=\bigcap _{u\in U}{\rm Ann}(u)\cap \bigcap _{v\in V}{\rm Ann}(v)=\bigcap _{x\in U\cup V}{\rm Ann}(x)={\rm Ann}(U\cup V)$. Clearly, if $U\subseteq V$, then ${\rm Ann}(V)\subseteq {\rm Ann}(U)$, hence ${\rm Ann}({\rm Ann}(U))\subseteq {\rm Ann}({\rm Ann}(V))$. Obviously, for any $x\in L$, ${\rm Ann}((x])={\rm Ann}(x)$. It is straightforward that, if $L$ is distributive, then ${\rm Ann}(U)$ is an ideal of $L$. Recall that $L$ is called a {\em compact lattice} iff all its elements are compact. Notice that, in a compact lattice $L$, the join of any non--empty family $U\subseteq L$ equals the join of a finite non--empty subfamily of $U$.

We shall denote by ${\cal B}(L)$ the set of the complemented elements of the bounded lattice $L$. If $L$ is distributive, then ${\cal B}(L)$ is the Boolean center of $L$, which is a Boolean sublattice of $L$. We shall call ${\cal B}(L)$ the {\em Boolean center of $L$} regardless of whether $L$ is distributive. We shall call $L$ a {\em Stone lattice} iff, for all $a\in L$, there exists an $e\in {\cal B}(L)$ such that ${\rm Ann}(a)=(e]$. We shall call $L$ a {\em strongly Stone lattice} iff, for all $U\subseteq L$, there exists an $e\in {\cal B}(L)$ such that ${\rm Ann}(U)=(e]$. Trivially, if $L$ is strongly Stone, then $L$ is Stone. Since $\displaystyle {\rm Ann}(U)=\bigcap _{u\in U}{\rm Ann}(u)$ for any $U\subseteq L$ and $\displaystyle \bigcap _{i\in I}(a_i]=(\bigvee _{i\in I}a_i]$ for any non--empty family $(a
_i)_{i\in I}\subseteq L$ having a join, it follows that the converse holds if ${\cal B}(L)$ is closed w.r.t. arbitrary joins.

Now let $M$ be a bounded lattice and $f:L\rightarrow M$ be a surjective lattice morphism. Then it is straightforward that the map $I\mapsto f(I)$ is a surjective lattice morphism from ${\rm Id}(L)$ to ${\rm Id}(M)$, which fulfills $f((a])=(f(a)]$ for all $a\in L$. Moreover, this lattice morphism preserves arbitrary joins. Indeed, if $(J_i)_{i\in I}$ is a family of ideals of $L$, then $\displaystyle f(\bigvee _{i\in I}J_i)\supseteq f(\bigcup _{i\in I}J_i)=\bigcup _{i\in I}f(J_i)$, thus $\displaystyle f(\bigvee _{i\in I}J_i)\supseteq \bigvee _{i\in I}f(J_i)$. Now, if $\displaystyle x\in \bigvee _{i\in I}J_i$, then $x\leq x_1\vee \ldots \vee x_n$ for some $n\in \N $, $i_1,\ldots ,i_n\in I$ and $x_1\in J_{i_1},\ldots ,x_n\in J_{i_n}$. Then $\displaystyle f(x)\leq f(x_1)\vee \ldots \vee f(x_n)\in f(J_{i_1})\vee \ldots \vee f(J_{i_n})\subseteq \bigvee _{i\in I}f(J_i)$, thus $\displaystyle f(x)\in \bigvee _{i\in I}f(J_i)$, hence $\displaystyle f(\bigvee _{i\in I}J_i)\subseteq \bigvee _{i\in I}f(J_i)$. Therefore $\displaystyle f(\bigvee _{i\in I}J_i)=\bigvee _{i\in I}f(J_i)$.

For the arbitrary bounded lattice $L$, let us denote by: ${\cal A}nn(L)=\{{\rm Ann}_L(U)\ |\ U\subseteq L\}$, ${\rm PAnn}(L)=\{{\rm Ann}_L(a)\ |\ a\in L\}$, ${\rm P2Ann}(L)=\{{\rm Ann}_L({\rm Ann})_L(a))\ |\ a\in L\}$ and ${\rm 2Ann}(L)=\{{\rm Ann}_L({\rm Ann}_L(U))\ |\ U\subseteq L\}$. Regarding the conditions below on subsets of $L$, note that they are trivially fulfilled by $\emptyset $, because ${\rm Ann}_L(\emptyset )=L=(1]_L$ and ${\rm Ann}_L({\rm Ann}_L(\emptyset ))={\rm Ann}_L(L)=\{0\}=(0]_L$ and, clearly, $0,1\in {\cal B}(L)$. Here, I am specifying $L$ through indexes in the notations, for clarity. Let us consider the following conditions on $L$, where $\kappa $ is an arbitrary nonzero cardinality; note that $(1)_{m,L},(2)_{m,L},(3)_{m,L},(4)_{m,L},(5)_{m,L}$ express the conditions from \cite[Theorem $1$]{dav} in a way that makes sense without $L$ being assumed distributive:

\begin{tabular}{ll}
$(1)_{\kappa ,L}$ & for each $U\subseteq L$ with $|U|\leq \kappa $, there exists an $e\in {\cal B}(L)$ such that ${\rm Ann}_L(U)=(e]_L$;\\ 
$(1)_{<\infty ,L}$ & for each finite $U\subseteq L$, there exists an $e\in {\cal B}(L)$ such that ${\rm Ann}_L(U)=(e]_L$;\\ 
$(1)_L$ & $L$ is a strongly Stone lattice;\end{tabular}

\begin{tabular}{ll}
$(2)_{\kappa ,L}$ & $L$ is a Stone lattice and ${\cal B}(L)$ is a $\kappa $--complete Boolean sublattice of $L$;\\ 
$(2)_{<\infty ,L}$ & $L$ is a Stone lattice and ${\cal B}(L)$ is a Boolean sublattice of $L$;\\ 
$(2)_L$ & $L$ is a Stone lattice and ${\cal B}(L)$ is a complete Boolean sublattice of $L$;\end{tabular}

\begin{tabular}{ll}
$(3)_{\kappa ,L}$ & ${\rm P2Ann}(L)$ is a $\kappa $--complete Boolean sublattice of ${\rm Id}(L)$ such that\\ 
& $a\mapsto {\rm Ann}_L({\rm Ann}_L(a))$ is a lattice morphism from $L$ to ${\rm P2Ann}(L)$;\\ 
$(3)_{<\infty ,L}$ & ${\rm P2Ann}(L)$ is a Boolean sublattice of ${\rm Id}(L)$ such that\\ 
& $a\mapsto {\rm Ann}_L({\rm Ann}_L(a))$ is a lattice morphism from $L$ to ${\rm P2Ann}(L)$;\\ 
$(3)_L$ & ${\rm P2Ann}(L)$ is a complete Boolean sublattice of ${\rm Id}(L)$ such that\\ 
& $a\mapsto {\rm Ann}_L({\rm Ann}_L(a))$ is a lattice morphism from $L$ to ${\rm P2Ann}(L)$;\end{tabular}

\begin{tabular}{ll}
$(4)_{\kappa ,L}$ & for all $a,b\in L$, ${\rm Ann}_L(a\wedge b)=({\rm Ann}_L(a)\cup {\rm Ann}_L(b)]_L$, and, for each $U\subseteq L$ with $|U|\leq \kappa $,\\ 
& there exists an $x\in L$ such that ${\rm Ann}_L({\rm Ann}_L(U))={\rm Ann}_L(x)$;\\ 
$(4)_{<\infty ,L}$ & for all $a,b\in L$, ${\rm Ann}_L(a\wedge b)=({\rm Ann}_L(a)\cup {\rm Ann}_L(b)]_L$, and, for each finite $U\subseteq L$,\\ 
& there exists an $x\in L$ such that ${\rm Ann}_L({\rm Ann}_L(U))={\rm Ann}_L(x)$;\\ 
$(4)_L$ & for all $a,b\in L$, ${\rm Ann}_L(a\wedge b)=({\rm Ann}_L(a)\cup {\rm Ann}_L(b)]_L$, and, for each $U\subseteq L$,\\ 
& there exists an $x\in L$ such that ${\rm Ann}_L({\rm Ann}_L(U))={\rm Ann}_L(x)$;\\ 
$(iv)_L$ & for all $a,b\in L$, ${\rm Ann}_L(a\wedge b)=({\rm Ann}_L(a)\cup {\rm Ann}_L(b)]_L$;\end{tabular}

\begin{tabular}{ll}
$(5)_{\kappa ,L}$ & for each $U\subseteq L$ with $|U|\leq \kappa $, $({\rm Ann}_L(U)\cup {\rm Ann}_L({\rm Ann}_L(U))]_L=L$;\\ 
$(5)_{<\infty ,L}$ & for each finite $U\subseteq L$, $({\rm Ann}_L(U)\cup {\rm Ann}_L({\rm Ann}_L(U))]_L=L$;\\ 
$(5)_L$ & for each $U\subseteq L$, $({\rm Ann}_L(U)\cup {\rm Ann}_L({\rm Ann}_L(U))]_L=L$.\end{tabular}

\begin{theorem}{\rm \cite[Theorem $1$]{dav}} If $L$ is a bounded distributive lattice, then the conditions $(1)_{m,L}$, $(2)_{m,L}$, $(3)_{m,L}$, $(4)_{m,L}$ and $(5)_{m,L}$ are equivalent.\label{davey}\end{theorem}

\begin{remark}{\rm \cite{eu,eu7}} Notice that, if we denote by $L^{\prime }$ the dual of the bounded lattice $L$, then the duals of conditions $(1)_{\kappa ,L}$ through $(5)_L$ above are simply conditions $(1)_{\kappa ,L^{\textstyle \prime }}$ through $(5)_{L^{\textstyle \prime }}$, respectively, that, with respect to $L$, can be expressed through co--anihilators and generated filters (see also \cite{eu,eu7}). In the case when $L$ is a bounded distributive lattice, so is $L^{\prime}$, thus conditions $(1)_{m,L^{\textstyle \prime }}$, $(2)_{m,L^{\textstyle \prime }}$, $(3)_{m,L^{\textstyle \prime }}$, $(4)_{m,L^{\textstyle \prime }}$ and $(5)_{m,L^{\textstyle \prime }}$ are equivalent, as well. And, of course, all the following properties on $L$ in this paper hold for $L^{\prime }$, too.

Let us also note that, for any $i\in \overline{1,5}$, we have the following: for any nonzero cardinalities $\kappa ,\mu $ such that $\kappa \leq \mu $, $(i)_{\mu ,L}$ implies $(i)_{\kappa ,L}$, and thus $(i)_{\mu ,L}$ is equivalent to $(i)_{\nu ,L}$ being valid for all nonzero cardinalities $\nu \leq \mu $; $(i)_L$ is equivalent to $(i)_{\kappa ,L}$ being valid for all nonzero cardinalities $\kappa $, as well as to $(i)_{\kappa ,L}$ being valid for all nonzero cardinalities $\kappa $ greater than a cardinality $\mu $; $(i)_{<\infty ,L}$ is equivalent to $(i)_{\kappa ,L}$ being valid for all finite nonzero cardinalities $\kappa $, as well as to $(i)_{\kappa ,L}$ being valid for all finite nonzero cardinalities $\kappa $ greater than a finite cardinality $\mu $. By the above and Theorem \ref{davey}, we get that, if $L$ is a bounded distributive lattice, then conditions $(1)_L$, $(2)_L$, $(3)_L$, $(4)_L$ and $(5)_L$ are equivalent.

Since, in any bounded lattice $L$, $\displaystyle {\rm Ann}_L(U)=\bigcap _{u\in U}{\rm Ann}_L(u)$ for each $U\subseteq L$ and, for any family $(x_i)_{i\in I}\subseteq L$ having a meet, $\displaystyle \bigcap _{i\in I}(x_i]_L=(\bigwedge _{i\in I}x_i]_L$, it follows that, for any bounded lattice $L$, $(2)_L$ implies $(1)_L$. The converse holds if ${\cal B}(L)$ is a complete Boolean sublattice of $L$, case in which, therefore, $L$ is Stone iff $L$ is strongly Stone, that is: for any finite nonzero cardinality $\kappa $, $(1)_L$ is equivalent to $(1)_{\kappa ,L}$ and thus also to $(1)_{<\infty ,L}$. Also, for any family $(U_i)_{i\in I}$ of subsets of $L$: $\displaystyle {\rm Ann}_L(\bigcup _{i\in I}U_i)=\bigcap _{u\in \bigcup _{i\in I}U_i}{\rm Ann}_L(u)=\bigcap _{i\in I}\bigcap _{u\in U_i}{\rm Ann}_L(u)=\bigcap _{i\in I}{\rm Ann}_L(U_i)$.

Now assume that $L$ is distributive. Then, for all $n\in \N ^*$ and all $u_1,\ldots ,u_n\in L$, $\displaystyle \bigcap _{i=1}^n{\rm Ann}_L(u_i)={\rm Ann}_L(\{u_1,\ldots ,$\linebreak $\displaystyle u_n\})={\rm Ann}_L(\bigvee _{i=1}^nu_i)$. Indeed, for all $k\in \overline{1,n}$, $\displaystyle u_k\leq \bigvee _{i=1}^nu_i$, hence $\displaystyle {\rm Ann}_L(\bigvee _{i=1}^nu_i)\subseteq {\rm Ann}_L(u_k)$, therefore $\displaystyle {\rm Ann}_L(\bigvee _{i=1}^nu_i)\subseteq \bigcap _{i=1}^n{\rm Ann}_L(u_i)$. If $\displaystyle x\in \bigcap _{i=1}^n{\rm Ann}_L(u_i)$, then $\displaystyle x\wedge (\bigvee _{i=1}^nu_i)=\bigvee _{i=1}^n(x\wedge u_i)=0$, thus $\displaystyle x\in {\rm Ann}_L(\bigvee _{i=1}^nu_i)$, so the converse inclusion holds, as well. Therefore ${\rm FAnn}(L)={\rm PAnn}(L)$, from which it is easy to see that, for any $i\in \overline{1,5}$ and any finite nonzero cardinality $\kappa $, $(i)_{\kappa ,L}$ is equivalent to $(i)_{<\infty ,L}$. It is immediate that, for any finite nonzero cardinality $\kappa $, $(1)_{\kappa ,L}$, $(2)_{\kappa ,L}$, $(3)_{\kappa ,L}$, $(4)_{\kappa ,L}$ and $(5)_{\kappa ,L}$ are equivalent.

If, moreover, $L$ is a frame, then, for any $\emptyset \neq U\subseteq L$, $\displaystyle \bigcap _{u\in U}{\rm Ann}_L(u)={\rm Ann}_L(U)={\rm Ann}_L(\bigvee _{u\in U}u)$. Indeed, for all $a\in U$, $\displaystyle a\leq \bigvee _{u\in U}u$, hence $\displaystyle {\rm Ann}_L(\bigvee _{u\in U}u)\subseteq {\rm Ann}_L(a)$, therefore $\displaystyle {\rm Ann}_L(\bigvee _{u\in U}u)\subseteq \bigcap _{u\in U}{\rm Ann}_L(u)$. If $\displaystyle x\in \bigcap _{u\in U}{\rm Ann}_L(u)$, then $\displaystyle x\wedge (\bigvee _{u\in U}u)=\bigvee _{u\in U}(x\wedge u)=0$, thus $\displaystyle x\in {\rm Ann}_L(\bigvee _{u\in U}u)$, so the converse inclusion holds, as well. Therefore, in a frame $L$, ${\cal A}nn(L)={\rm PAnn}(L)$, from which it is easy to see that, for any $i\in \overline{1,5}$ and any finite nonzero cardinality $\kappa $, $(i)_{\kappa ,L}$ is equivalent to $(i)_L$, and it also follows that ${\rm 2Ann}(L)={\rm P2Ann}(L)\subseteq {\cal A}nn(L)={\rm PAnn}(L)$, so that the second part of condition $(4)_{\kappa ,L}$ is fulfilled for any nonzero cardinality $\kappa $, and thus $(4)_{\kappa ,L}$, $(4)_{<\infty ,L}$, $(4)_L$ and $(iv)_L$ are equivalent.\label{dualdavey}\end{remark}

\begin{corollary}If $L$ is a bounded distributive lattice, then:\begin{itemize}
\item conditions $(1)_L$, $(2)_L$, $(3)_L$, $(4)_L$ and $(5)_L$ are equivalent;
\item for any nonzero cardinality $\kappa $, $(1)_{\kappa ,L}$, $(2)_{\kappa ,L}$, $(3)_{\kappa ,L}$, $(4)_{\kappa ,L}$ and $(5)_{\kappa ,L}$ are equivalent;
\item for any finite nonzero cardinality $\kappa $, $(1)_{\kappa ,L}$, $(2)_{\kappa ,L}$, $(3)_{\kappa ,L}$, $(4)_{\kappa ,L}$, $(5)_{\kappa ,L}$, $(1)_{<\infty ,L}$, $(2)_{<\infty ,L}$, $(3)_{<\infty ,L}$, $(4)_{<\infty ,L}$ and $(5)_{<\infty ,L}$ are equivalent;
\item if, moreover, $L$ is a frame, then, for any nonzero cardinality $\kappa $ and any $h,i,j\in \overline{1,5}$, $(iv)_L$, $(h)_{\kappa ,L}$, $(i)_{<\infty ,L}$ and $(j)_L$ are equivalent; in particular, $(1)_L$ is equivalent to $(2)_{<\infty ,L}$, so that: $L$ is a Stone lattice iff $L$ is a strongly Stone lattice.\end{itemize}\label{moredavey}\end{corollary}

Throughout this paper, all algebras shall be non--empty, ${\cal C}$ shall be a semi--degenerate congruence--modular equational class of algebras of the same type and $A$ shall be an arbitrary algebra from ${\cal C}$. $({\rm Con}(A),\vee ,\cap ,\Delta _A,\nabla _A)$ shall be the complete modular lattice of the congruences of $A$, with $\Delta _A=\{(a,a)\ |\ a\in A\}$ and $\nabla _A=A^2$, so that the set of the {\em proper congruences} of $A$ is ${\rm Con}(A)\setminus \{\nabla _A\}$. $[\cdot ,\cdot ]_A:({\rm Con}(A))^2\rightarrow {\rm Con}(A)$ shall be the commutator of $A$. Recall that $[\cdot ,\cdot ]_A$ is commutative, smaller than the intersection, increasing in both arguments and distributive in both arguments with respect to arbitrary joins \cite{fremck}. Following \cite{fremck}, if $\phi $ is a proper congruence of $A$, then we call $\phi $ a {\em prime congruence} iff, for all $\theta ,\zeta \in {\rm Con}(A)$, $[\theta ,\zeta ]_A\subseteq \phi $ implies $\theta \subseteq \phi $ or $\zeta \subseteq \phi $. We shall denote by ${\rm Spec}(A)$ the set of the prime congruences of $A$.

Since ${\cal C}$ is semi--degenerate, it follows that ${\cal C}$ has no skew congruences \cite[Theorem 8.5, p. 85]{fremck} and, for any member $M$ of ${\cal C}$, $\nabla _M$ is a compact congruence of $M$ \cite{koll} and each proper congruence of $M$ is included in a prime congruence \cite[Theorem $5.3$]{agl}, thus, if $M$ is {\em non--trivial}, that is $|M|>1$, then ${\rm Spec}(M)$ is non--empty. Recall that the compact congruences of an algebra are exactly its finitely generated congruences.

Following \cite{retic}, for all $\theta \in {\rm Con}(A)$, we shall denote by $\rho _A(\theta )$ the {\em radical of $\theta $}: $\rho _A(\theta )=\bigcap \{\phi \in {\rm Spec}(A)\ |\ \theta \subseteq \phi \}$, and by $\equiv _A$ the binary relation on ${\rm Con}(A)$ defined by: for any $\theta ,\zeta \in {\rm Con}(A)$, $\theta \equiv _A\zeta $ iff $\rho _A(\theta )=\rho _A(\zeta )$. We have proven, in \cite{retic}, that $\equiv _A$ is an equivalence on ${\rm Con}(A)$, and we have denoted by $\lambda _A:{\rm Con}(A)\rightarrow {\rm Con}(A)/_{\textstyle \equiv _A}$ the canonical surjection. Moreover, we have proven that $\equiv _A$ is a congruence of the lattice ${\rm Con}(A)$, thus $({\rm Con}(A)/_{\textstyle \equiv _A},\vee ,\wedge ,\lambda _A(\Delta _A),\lambda _A(\nabla _A))$ is a bounded lattice and $\lambda _A$ is a surjective lattice morphism, where $\lambda _A(\theta )\vee \lambda _A(\zeta )=\lambda _A(\theta \vee \zeta )$ and $\lambda _A(\theta )\wedge \lambda _A(\zeta )=\lambda _A(\theta \cap \zeta )$ for all $\theta ,\zeta \in {\rm Con}(A)$. We have also proven that $\lambda _A(\theta \cap \zeta )=\lambda _A([\theta ,\zeta ]_A)$ for all $\theta ,\zeta \in {\rm Con}(A)$, from which, by using the distributivity of the commutator with respect to the join, it immediately follows that the bounded lattice ${\rm Con}(A)/_{\textstyle \equiv _A}$ is distributive. Moreover, since ${\rm Con}(A)$ is a complete lattice and $[\cdot ,\cdot ]_A$ is distributive w.r.t. arbitrary joins, it follows that ${\rm Con}(A)/_{\textstyle \equiv _A}$ is a complete lattice in which the meet is distributive w.r.t. arbitrary joins, that is ${\rm Con}(A)/_{\textstyle \equiv _A}$ is a frame (see also \cite{retic}). Hence, from Corollary \ref{moredavey}, we obtain:

\begin{corollary} For any nonzero cardinality $\kappa $ and any $h,i,j\in \overline{1,5}$, $(h)_{\kappa ,{\rm Con}(A)/_{\textstyle \equiv _A}}$, $(i)_{<\infty ,{\rm Con}(A)/_{\textstyle \equiv _A}}$ and $(j)_{{\rm Con}(A)/_{\textstyle \equiv _A}}$ are equivalent.\label{cordav}\end{corollary}

We have proven in \cite{retic} that:\begin{itemize}
\item for any $\theta \in {\rm Con}(A)$, $\lambda _A(\theta )=\lambda _A(\nabla
_A)$ iff $\theta =\nabla
_A$; indeed, $\lambda _A(\theta )=\lambda _A(\nabla
_A)$ iff $\rho _A(\theta )=\rho _A(\nabla
_A)=\bigcap \emptyset =\nabla _A$ iff no prime congruence of $A$ includes $\theta $ iff $\theta =\nabla _A$;
\item ${\cal B}({\rm Con}(A))$ is a Boolean sublattice of ${\rm Con}(A)$, in which the commutator coincides to the intersection.\end{itemize}

We call $A$ a {\em semiprime algebra} iff $\rho _A(\Delta _A)=\Delta _A$. If ${\cal C}$ is congruence--distributive, then $\rho _A(\theta )=\theta $ for all $\theta \in {\rm Con}(A)$, thus $A$ is semiprime (see, also, {\rm \cite{retic}}).

Throughout the rest of this paper, the algebra $A$ shall be semiprime. Then, as we have proven in \cite{retic}:\begin{itemize}
\item $(\circ )\quad $ ${\cal B}({\rm Con}(A)/_{\textstyle \equiv _A})={\cal B}({\rm Con}(A))/_{\textstyle \equiv _A}$;
\item $(1^{\circ })\quad $ $\lambda _A\mid _{{\cal B}({\rm Con}(A))}:{\cal B}({\rm Con}(A))\rightarrow {\cal B}({\rm Con}(A)/_{\textstyle \equiv _A})={\cal B}({\rm Con}(A))/_{\textstyle \equiv _A}$ is a Boolean isomorphism;
\item $(2^{\circ })\quad $ for all $\theta \in {\rm Con}(A)$: $\lambda _A(\theta )\in {\cal B}({\rm Con}(A)/_{\textstyle \equiv _A})$ iff $\theta \in {\cal B}({\rm Con}(A))$;
\item $(3^{\circ })\quad $ for all $\theta \in {\rm Con}(A)$, $\lambda _A(\theta )=\lambda _A(\Delta _A)$ iff $\theta =\Delta _A$; indeed, $\lambda _A(\theta )=\lambda _A(\Delta _A)$ implies $\theta \subseteq \rho _A(\theta )=\rho _A(\Delta _A)=\Delta _A$, thus $\theta =\Delta _A$; the converse implication is trivial;
\item $(4^{\circ })\quad $ for all $\theta ,\zeta \in {\rm Con}(A)$, $[\theta ,\zeta ]_A=\Delta _A$ iff $\theta \cap \zeta =\Delta _A$; indeed, by $(3^{\circ })$: $[\theta ,\zeta ]_A=\Delta _A$ iff $\lambda _A([\theta ,\zeta ]_A)=\lambda _A(\Delta _A)$ iff $\lambda _A(\theta )\wedge \lambda _A(\zeta )=\lambda _A(\Delta _A)$ iff $\lambda _A(\theta \cap \zeta )=\lambda _A(\Delta _A)$ iff $\theta \cap \zeta =\Delta _A$.\end{itemize}

I shall make repeated use of the surjectivity of $\lambda _A:{\rm Con}(A)\rightarrow {\rm Con}(A)/_{\textstyle \equiv _A}$, without mentioning it; the same goes for the remarks from this paper and the results I am recalling from \cite{retic}, excepting $(1^{\circ })$, $(2^{\circ })$, $(3^{\circ })$ and $(4^{\circ })$.

By Corollary \ref{moredavey}, in the particular case when ${\rm Con}(A)$ is distributive, conditions $(1)_{{\rm Con}(A)},\ldots ,(5)_{{\rm Con}(A)}$ are equivalent, if $\kappa $ is a nonzero cardinality, then conditions $(1)_{\kappa ,{\rm Con}(A)},\ldots ,(5)_{\kappa ,{\rm Con}(A)}$ are equivalent, and, if $\kappa $ is finite, then conditions $(1)_{\kappa ,{\rm Con}(A)},\ldots ,(5)_{\kappa ,{\rm Con}(A)},(1)_{<\infty ,{\rm Con}(A)},\ldots ,(5)_{<\infty ,{\rm Con}(A)}$ are equivalent. An example of a semi--degenerate congruence--modular variety which is not congruence--distributive is the variety of commutative unitary rings \cite{kap}.

Throughout the rest of this section, unless there is danger of confusion, all annihilators of elements or subsets of ${\rm Con}(A)$, respectively ${\rm Con}(A)/_{\textstyle \equiv _A}$,  shall be in the lattice ${\rm Con}(A)$, respectively ${\rm Con}(A)/_{\textstyle \equiv _A}$, and the same shall go for generated ideals.

\begin{remark} By $(4^{\circ })$, for all $\theta \in {\rm Con}(A)$ and all $\Omega \subseteq {\rm Con}(A)$, ${\rm Ann}(\theta )=\{\zeta \in {\rm Con}(A)\ |\ \theta \cap \zeta =\Delta _A\}=\{\zeta \in {\rm Con}(A)\ |\ [\theta ,\zeta ]_A=\Delta _A\}$ and ${\rm Ann}(\Omega )=\{\zeta \in {\rm Con}(A)\ |\ (\forall \, \omega \in \Omega )\, (\omega \cap \zeta =\Delta _A)\}=\{\zeta \in {\rm Con}(A)\ |\ (\forall \, \omega \in \Omega )\, ([\omega ,\zeta ]_A=\Delta _A)\}$.\end{remark}

\begin{remark} Any annihilator of ${\rm Con}(A)$ is an ideal of ${\rm Con}(A)$. Indeed, let $\omega \in {\rm Con}(A)$. Then $\Delta _A\cap \omega =\Delta _A$, thus $\Delta _A\in {\rm Ann}(\omega )$, so ${\rm Ann}(\omega )\neq \emptyset $. Now let $\theta ,\zeta \in {\rm Con}(A)$. If $\zeta \in {\rm Ann}(\omega )$ and $\theta \subseteq \zeta $, then $\theta \cap \omega \subseteq \zeta \cap \omega =\Delta _A$, thus $\theta \in {\rm Ann}(\omega )$. If $\theta ,\zeta \in {\rm Ann}(\omega )$, then $[\omega ,\theta \vee \zeta ]_A=[\omega ,\theta ]_A\vee [\omega ,\zeta ]_A=\Delta _A\vee \Delta _A=\Delta _A$, hence $\theta \vee \zeta \in {\rm Ann}(\omega )$. Therefore ${\rm Ann}(\omega )\in {\rm Id}({\rm Con}(A))$. Hence, for any $\Omega \subseteq {\rm Con}(A)$, $\displaystyle {\rm Ann}(\Omega )=\bigcap _{\omega \in \Omega }{\rm Ann}(\omega )\in {\rm Id}({\rm Con}(A))$.\label{2.6}\end{remark}

\begin{remark} Of course, the direct image of $\lambda _A:{\rm Con}(A)\rightarrow {\rm Con}(A)/_{\textstyle \equiv _A}$ preserves arbitrary unions of subsets of ${\rm Con}(A)$. It also preserves arbitrary intersections, because, if $(\Omega _i)_{i\in I}$ is a family of subsets of ${\rm Con}(A)$, then $\displaystyle \lambda _A(\bigcap _{i\in I}\Omega _i)=(\bigcap _{i\in I}\Omega _i)/_{\textstyle \equiv _A}=\bigcap _{i\in I}(\Omega _i/_{\textstyle \equiv _A})=\bigcap _{i\in I}\lambda _A(\Omega _i)$. Since $\lambda _A:{\rm Con}(A)\rightarrow {\rm Con}(A)/_{\textstyle \equiv _A}$ is a surjective lattice morphism, it follows that the map $I\mapsto \lambda _A(I)$ is a surjective lattice morphism from ${\rm Id}({\rm Con}(A))$ to ${\rm Id}({\rm Con}(A)/_{\textstyle \equiv _A})$ which fulfills $\lambda _A((\theta ])=(\lambda _A(\theta )]$ for all $\theta \in {\rm Con}(A)$ and preserves arbitrary joins.\label{2.7}\end{remark}

\begin{lemma}\begin{enumerate}
\item\label{annidgen1} If $L$ is a bounded distributive lattice, then ${\rm Ann}_L(U)={\rm Ann}_L((U]_L)$ for all $U\subseteq L$.
\item\label{annidgen2} For all $\Omega \subseteq {\rm Con}(A)$, ${\rm Ann}(\Omega )={\rm Ann}((\Omega ])$.\end{enumerate}\label{annidgen}\end{lemma}

\begin{proof} (\ref{annidgen1}) $U\subseteq (U]_L$, thus ${\rm Ann}_L((U]_L)\subseteq {\rm Ann}_L(U)$. If $x\in {\rm Ann}_L(U)$ and $u\in (U]_L$, then $u\leq u_1\vee \ldots \vee u_n$ for some $n\in \N $ and $u_1,\ldots ,u_n\in U$, thus $x\wedge u\leq x\wedge (u_1\vee \ldots \vee u_n)=(x\wedge u_1)\vee \ldots \vee (x\wedge u_n)=0\vee \ldots \vee 0=0$, thus $x\in {\rm Ann}_L((U]_L)$ since $u$ is arbitrary in $(U]_L$, so the converse inclusion holds, as well.

\noindent (\ref{annidgen2}) $\Omega \subseteq (\Omega ]$, thus ${\rm Ann}((\Omega ])\subseteq {\rm Ann}(\Omega )$. Now let $\theta \in {\rm Ann}(\Omega )$ and $\omega \in (\Omega ]$, so that $\omega \subseteq \omega _1\vee \ldots \vee \omega _n$ for some $n\in \N $ and $\omega _1,\ldots , \omega _n\in \Omega $, thus $[\theta ,\omega ]_A\subseteq [\theta ,\omega _1\vee \ldots \vee \omega _n]_A=[\theta ,\omega _1]_A\vee \ldots \vee [\theta ,\omega _n]_A=\Delta _A\vee \ldots \vee \Delta _A=\Delta _A$, thus $\theta \in {\rm Ann}((\Omega ])$ since $\omega $ is arbitrary in $(\Omega ]$, so the converse inclusion holds, as well.\end{proof}

\begin{lemma} If $L$ is a bounded distributive lattice, then:\begin{itemize}
\item for any family $(I_k)_{k\in K}\subseteq {\rm Id}(L)$, $\displaystyle \bigcap _{k\in K}{\rm Ann}_L(I_k)={\rm Ann}_L(\bigvee _{k\in K}I_k)$;
\item for any family $(a_k)_{k\in K}\subseteq L$ having a join, $\displaystyle \bigcap _{k\in K}{\rm Ann}_L(a_k)={\rm Ann}_L(\bigvee _{k\in K}a_k)$.\end{itemize}\label{semilatl}\end{lemma}

\begin{proof} By Lemma \ref{annidgen}, (\ref{annidgen1}), $\displaystyle \bigcap _{k\in K}{\rm Ann}_L(I_k)={\rm Ann}_L(\bigcup _{k\in K}I_k)={\rm Ann}_L((\bigcup _{k\in K}I_k]_L)={\rm Ann}_L(\bigvee _{k\in K}I_k)$, hence, if $\displaystyle \bigvee _{k\in K}a_k$ exists, then $\displaystyle \bigcap _{k\in K}{\rm Ann}_L(a_k)=\bigcap _{k\in K}{\rm Ann}_L((a_k]_L)={\rm Ann}_L(\bigvee _{k\in K}(a_k]_L)={\rm Ann}_L((\bigvee _{k\in K}a_k]_L)={\rm Ann}_L(\bigvee _{k\in K}a_k)$.\end{proof}

\begin{lemma}\begin{itemize}
\item For any family $(I_k)_{k\in K}\subseteq {\rm Id}({\rm Con}(A))$, $\displaystyle \bigcap _{k\in K}{\rm Ann}(I_k)={\rm Ann}(\bigvee _{k\in K}I_k)$.
\item For any family $(\theta _k)_{k\in K}\subseteq {\rm Con}(A)$, $\displaystyle \bigcap _{k\in K}{\rm Ann}(\theta _k)={\rm Ann}(\bigvee _{k\in K}\theta _k)$.\end{itemize}\label{semilata}\end{lemma}

\begin{proof} Same as the proof of Lemma \ref{semilatl}, but using (\ref{annidgen2}) from Lemma \ref{annidgen} instead of (\ref{annidgen1}).\end{proof}

\begin{lemma}\begin{itemize}
\item For all $\theta \in {\rm Con}(A)$, ${\rm Ann}(\lambda _A(\theta ))=\lambda _A({\rm Ann}(\theta ))$.
\item For all $\Omega \in {\rm Con}(A)$, ${\rm Ann}(\lambda _A(\Omega ))=\lambda _A({\rm Ann}(\Omega ))$.\end{itemize}\label{annlambda}\end{lemma}

\begin{proof} ${\rm Ann}(\lambda _A(\theta ))=\{\lambda _A(\zeta )\ |\ \zeta \in {\rm Con}(A),\lambda _A(\theta )\wedge \lambda _A(\zeta )=\lambda _A(\Delta _A)\}=\{\lambda _A(\zeta )\ |\ \zeta \in {\rm Con}(A),\lambda _A(\theta \cap \zeta )=\lambda _A(\Delta _A)\}=\{\lambda _A(\zeta )\ |\ \zeta \in {\rm Con}(A),\theta \cap \zeta =\Delta _A\}=\lambda _A({\rm Ann}(\theta ))$, by $(3^{\circ })$. Therefore $\displaystyle \lambda _A({\rm Ann}(\Omega ))=\lambda _A(\bigcap _{\omega \in \Omega }{\rm Ann}(\omega ))=\bigcap _{\omega \in \Omega }\lambda _A({\rm Ann}(\omega ))=\bigcap _{\omega \in \Omega }{\rm Ann}(\lambda _A(\omega ))=\bigcap _{x\in \lambda _A(\Omega )}{\rm Ann}(x)={\rm Ann}(\lambda _A(\Omega ))$.\end{proof}

\begin{lemma} For all $\theta \in {\rm Con}(A)$ and $\alpha \in {\cal B}({\rm Con}(A))$, $\lambda _A(\theta )\leq \lambda _A(\alpha )$ iff $\theta \subseteq \alpha $.\label{pidbool}\end{lemma}

\begin{proof} Assume that $\lambda _A(\theta )\leq \lambda _A(\alpha )$, so that $\lambda _A(\theta \vee \alpha )=\lambda _A(\theta )\vee \lambda _A(\alpha )=\lambda _A(\alpha )\in \lambda _A({\cal B}({\rm Con}(A)))={\cal B}({\rm Con}(A)/_{\textstyle \equiv _A})$ by $(1^{\circ })$, hence $\theta \vee \alpha \in {\cal B}({\rm Con}(A))$ by $(2^{\circ })$, and thus $\theta \vee \alpha =\alpha $, again by $(1^{\circ })$, so $\theta \subseteq \alpha $. The fact that $\lambda _A$ is order--preserving proves the converse implication.\end{proof}

\begin{lemma} For all $\Gamma ,\Omega \subseteq {\rm Con}(A)$, $\gamma \in {\rm Con}(A)$ and $\alpha \in {\cal B}({\rm Con}(A))$:\begin{enumerate}
\item\label{annegal1} $\lambda _A(\gamma )\in \lambda _A({\rm Ann}(\Omega ))$ iff $\gamma \in {\rm Ann}(\Omega )$;
\item\label{annegal2} $\lambda _A(\Gamma )\subseteq \lambda _A({\rm Ann}(\Omega ))$ iff $\Gamma \subseteq {\rm Ann}(\Omega )$;
\item\label{annegal3} $\lambda _A({\rm Ann}(\Gamma ))=\lambda _A({\rm Ann}(\Omega ))$ iff ${\rm Ann}(\Gamma )={\rm Ann}(\Omega )$;
\item\label{annegal4} $\lambda _A((\alpha ])=\lambda _A({\rm Ann}(\Omega ))$ iff $(\alpha ]={\rm Ann}(\Omega )$.\end{enumerate}\label{annegal}\end{lemma}

\begin{proof} (\ref{annegal1}) Assume that $\lambda _A(\gamma )\in \lambda _A({\rm Ann}(\Omega ))={\rm Ann}(\lambda _A(\Omega ))$ according to Lemma \ref{annlambda}, so that, for all $\omega \in \Omega $, $\lambda _A(\gamma \cap \omega )=\lambda _A(\gamma )\wedge \lambda _A(\omega )=\lambda _A(\Delta _A)$, thus, by $(3^{\circ })$, $\gamma \cap \omega =\Delta _A$, hence $\gamma \in {\rm Ann}(\Omega )$. The converse implication is trivial.

\noindent (\ref{annegal2}) By (\ref{annegal1}).

\noindent (\ref{annegal3}) By (\ref{annegal2}), $\lambda _A({\rm Ann}(\Gamma ))=\lambda _A({\rm Ann}(\Omega ))$ iff $\lambda _A({\rm Ann}(\Gamma ))\subseteq \lambda _A({\rm Ann}(\Omega ))$ and $\lambda _A({\rm Ann}(\Omega ))\subseteq \lambda _A({\rm Ann}(\Gamma ))$ iff ${\rm Ann}(\Gamma )\subseteq {\rm Ann}(\Omega )$ and ${\rm Ann}(\Omega )\subseteq {\rm Ann}(\Gamma )$ iff ${\rm Ann}(\Gamma )={\rm Ann}(\Omega )$.

\noindent (\ref{annegal4}) Assume that $\lambda _A((\alpha ])=\lambda _A({\rm Ann}(\Omega ))$, which implies $(\alpha ]\subseteq {\rm Ann}(\Omega )$ by (\ref{annegal2}). Now let $\theta \in {\rm Ann}(\Omega )$, so that $\lambda _A(\theta )\in \lambda _A({\rm Ann}(\Omega ))=\lambda _A((\alpha ])=(\lambda _A(\alpha )]$, hence $\theta \in (\alpha ]$ by Lemma \ref{pidbool}, thus ${\rm Ann}(\Omega )\subseteq (\alpha ]$, therefore $(\alpha ]={\rm Ann}(\Omega )$. The converse implication is trivial.\end{proof}

\begin{proposition} For any nonzero cardinality $\kappa $, the properties $(1)_{\kappa ,{\rm Con}(A)}$ and $(1)_{\kappa ,{\rm Con}(A)/_{\scriptstyle \equiv _A}}$ are equivalent.\label{echiv(i)}\end{proposition}

\begin{proof} For the converse implication, assume that $(1)_{\kappa ,{\rm Con}(A)/_{\scriptstyle \equiv _A}}$ is satisfied and let $\emptyset \neq \Omega \subseteq {\rm Con}(A)$ having $|\Omega |\leq \kappa $. Then $|\lambda _A(\Omega )|\leq |\Omega |\leq \kappa $, hence, by $(1^{\circ })$, Lemma \ref{annlambda} and Lemma \ref{annegal}, (\ref{annegal4}), there exists an $\alpha \in {\cal B}({\rm Con}(A))$ such that $\lambda _A({\rm Ann}(\Omega ))={\rm Ann}(\lambda _A(\Omega ))=(\lambda _A(\alpha )]=\lambda _A((\alpha ])$, therefore ${\rm Ann}(\Omega )=(\alpha ]$.

For the direct implication, assume that $(1)_{\kappa ,{\rm Con}(A)}$ is satisfied and let $U\subseteq {\rm Con}(A)/_{\textstyle \equiv _A}$ with $|U|\leq \kappa $. For each $u\in U$, there exists an $\omega _u\in {\rm Con}(A)$ such that $\lambda _A(\omega _u)=u$. If we denote by $\Omega =\{\omega _u\ |\ u\in U\}\subseteq {\rm Con}(A)$, then $\lambda _A(\Omega )=U$ and $|\Omega |=|U|\leq \kappa $, hence, by Lemma \ref{annlambda} and $(1^{\circ })$, ${\rm Ann}(U)={\rm Ann}(\lambda _A(\Omega ))=\lambda _A({\rm Ann}(\Omega ))=\lambda _A((\alpha ])=(\lambda _A(\alpha )]$ for some $\alpha \in {\cal B}({\rm Con}(A))$, so that $\lambda _A(\alpha )\in {\cal B}({\rm Con}(A)/_{\textstyle \equiv _A})$.\end{proof}

\begin{proposition} ${\rm Con}(A)$ is a Stone lattice iff ${\rm Con}(A)/_{\textstyle \equiv _A}$ is a Stone lattice.\label{echivstone}\end{proposition}

\begin{proof} If ${\rm Con}(A)$ is a Stone lattice, then, for all $\theta \in {\rm Con}(A)$, there exists an $\alpha \in {\cal B}({\rm Con}(A))$ such that ${\rm Ann}(\theta )=(\alpha ]$, so that $\lambda _A(\alpha )\in {\cal B}({\rm Con}(A)/_{\textstyle \equiv _A})$ and ${\rm Ann}(\lambda _A(\theta ))=\lambda _A({\rm Ann}(\theta ))=\lambda _A((\alpha ])=(\lambda _A(\alpha )]$ by $(1^{\circ })$ and Lemma \ref{annlambda}, hence ${\rm Con}(A)/_{\textstyle \equiv _A}$ is a Stone lattice.

If ${\rm Con}(A)/_{\textstyle \equiv _A}$ is a Stone lattice, then, by $(1^{\circ })$, Lemma \ref{annlambda} and Lemma \ref{annegal}, (\ref{annegal4}), for all $\theta \in {\rm Con}(A)$, there exists an $\alpha \in {\cal B}({\rm Con}(A))$ such that $\lambda _A({\rm Ann}(\theta ))={\rm Ann}(\lambda _A(\theta ))=(\lambda _A(\alpha )]=\lambda _A((\alpha ])$, hence ${\rm Ann}(\theta )=(\alpha ]$, therefore ${\rm Con}(A)$ is a Stone lattice.\end{proof}

\begin{proposition} For any nonzero cardinality $\kappa $, the properties $(2)_{\kappa ,{\rm Con}(A)}$ and $(2)_{\kappa ,{\rm Con}(A)/_{\scriptstyle \equiv _A}}$ are equivalent.\label{echiv(ii)}\end{proposition}

\begin{proof} By Proposition \ref{echivstone} and $(1^{\circ })$, which ensures us that the Boolean algebras ${\cal B}({\rm Con}(A))$ and ${\cal B}({\rm Con}(A)/_{\textstyle \equiv _A})$ are isomorphic.\end{proof}

\begin{proposition} ${\rm Con}(A)$ is a strongly Stone lattice iff ${\rm Con}(A)/_{\textstyle \equiv _A}$ is a strongly Stone lattice.\label{echivtarestone}\end{proposition}

\begin{proof} Same as the proof of Proposition \ref{echivstone}.\end{proof}

\begin{remark} Since ${\rm Con}(A)/_{\textstyle \equiv _A}$ is a frame, from Corollary \ref{moredavey}, we get that: the lattice ${\rm Con}(A)/_{\textstyle \equiv _A}$ is Stone iff it is strongly Stone.\end{remark}

\begin{corollary} The lattice ${\rm Con}(A)$ is Stone iff ${\rm Con}(A)$ is strongly Stone iff ${\rm Con}(A)/_{\textstyle \equiv _A}$ is Stone iff ${\rm Con}(A)/_{\textstyle \equiv _A}$ is strongly Stone.\end{corollary}

Let us consider the following conditions:

\begin{tabular}{cl}
$(pann)_L$ & ${\rm PAnn}(L)$ is a sublattice of ${\rm Id}(L)$ such that the map $a\mapsto {\rm Ann}_L(a)$ is\\
& a lattice anti--morphism from $L$ to ${\rm PAnn}(L)$;\\ 
$(p2ann)_L$ & ${\rm P2Ann}(L)$ is a sublattice of ${\rm Id}(L)$ such that the map $a\mapsto {\rm Ann}_L({\rm Ann}_L(a))$ is\\
& a lattice morphism from $L$ to ${\rm P2Ann}(L)$.\end{tabular}

\begin{remark} Concerning the following results, recall that ${\rm Con}(A)/_{\textstyle \equiv _A}$ is a frame, and thus ${\rm PAnn}({\rm Con}(A)/_{\textstyle \equiv _A})$\linebreak $={\cal A}nn({\rm Con}(A)/_{\textstyle \equiv _A})$ and ${\rm P2Ann}({\rm Con}(A)/_{\textstyle \equiv _A})={\rm 2Ann}({\rm Con}(A)/_{\textstyle \equiv _A})$.\end{remark}

\begin{lemma}\begin{enumerate}
\item\label{annbij1} The map $P\mapsto \lambda _A(P)$ from ${\rm PAnn}({\rm Con}(A))$ to ${\rm PAnn}({\rm Con}(A)/_{\textstyle \equiv _A})$ is an order isomorphism.
\item\label{annbij2} For all $\theta ,\zeta \in {\rm Con}(A)$: ${\rm Ann}(\theta \cap \zeta )={\rm Ann}(\theta )\vee {\rm Ann}(\zeta )$ iff ${\rm Ann}(\lambda _A(\theta )\wedge \lambda _A(\zeta ))={\rm Ann}(\lambda _A(\theta ))\vee {\rm Ann}(\lambda _A(\zeta ))$.
\item\label{annbij3} $(pann)_{{\rm Con}(A)}$ is equivalent to $(pann)_{{\rm Con}(A)/_{\scriptstyle \equiv _A}}$.
\item\label{annbij4} If the equivalent conditions from (\ref{annbij3}) are fulfilled, then the map from (\ref{annbij1}) is a lattice isomorphism.\end{enumerate}\label{annbij}\end{lemma}

\begin{proof} (\ref{annbij1}) By Lemma \ref{annlambda}, this restriction of the direct image of $\lambda _A$ takes ${\rm Ann}(\theta )$ to ${\rm Ann}(\lambda _A(\theta ))$ for all $\theta \in {\rm Con}(A)$, thus it is well defined and surjective. By Lemma \ref{annegal}, (\ref{annegal3}), it is also injective. By Lemma \ref{annegal}, (\ref{annegal2}), for all $\theta ,\zeta \in {\rm Con}(A)$, ${\rm Ann}(\theta )\subseteq {\rm Ann}(\zeta )$ iff $\lambda _A({\rm Ann}(\theta ))\subseteq \lambda _A({\rm Ann}(\zeta ))$, therefore this bijection and its inverse preserve order. So this map is an order isomorphism, which is a restriction of the lattice morphism $I\mapsto \lambda _A(I)$ from ${\rm Id}({\rm Con}(A))$ to ${\rm Id}({\rm Con}(A)/_{\textstyle \equiv _A})$.

\noindent (\ref{annbij2}) and (\ref{annbij3}) By (\ref{annbij1}), the map $P\mapsto \lambda _A(P)$ from ${\rm PAnn}({\rm Con}(A))$ to ${\rm PAnn}({\rm Con}(A)/_{\textstyle \equiv _A})$ preserves all joins and so does its inverse. Lemmas \ref{semilata} and \ref{semilatl} show that ${\rm PAnn}({\rm Con}(A))$ and ${\rm PAnn}({\rm Con}(A)/_{\textstyle \equiv _A})$ always are inferior subsemilattices of ${\rm Id}({\rm Con}(A))$, respectively ${\rm Id}({\rm Con}(A)/_{\textstyle \equiv _A})$. From (\ref{annbij1}) it follows that ${\rm PAnn}({\rm Con}(A))$ is closed with respect to the join from the lattice ${\rm Id}({\rm Con}(A))$ iff ${\rm PAnn}({\rm Con}(A)/_{\textstyle \equiv _A})$ is closed with respect to the join from the lattice ${\rm Id}({\rm Con}(A)/_{\textstyle \equiv _A})$, and, if they are closed with respect to the join, then, for all $\theta ,\zeta \in {\rm Con}(A)$, ${\rm Ann}(\theta )\vee {\rm Ann}(\zeta )={\rm Ann}(\theta \cap \zeta )$ iff $\lambda _A({\rm Ann}(\theta )\vee {\rm Ann}(\zeta ))=\lambda _A({\rm Ann}(\theta \cap \zeta ))$, which in turn is equivalent to ${\rm Ann}(\lambda _A(\theta ))\vee {\rm Ann}(\lambda _A(\zeta ))={\rm Ann}(\lambda _A(\theta \cap \zeta ))$ by Lemma \ref{annlambda} and the fact that the direct image of $\lambda _A$ preserves the joins of ideals. Therefore ${\rm PAnn}({\rm Con}(A))$ is a sublattice of ${\rm Id}({\rm Con}(A))$ iff ${\rm PAnn}({\rm Con}(A)/_{\textstyle \equiv _A})$ is a sublattice of ${\rm Id}({\rm Con}(A)/_{\textstyle \equiv _A})$, with the expressions for the joins above.

\noindent (\ref{annbij4}) By (\ref{annbij1}).\end{proof}

\begin{lemma}\begin{enumerate}
\item\label{echiv(iii)nobm1} The map $Q\mapsto \lambda _A(Q)$ from ${\rm P2Ann}({\rm Con}(A))$ to ${\rm P2Ann}({\rm Con}(A)/_{\textstyle \equiv _A})$ is an order isomorphism.
\item\label{echiv(iii)nobm2} $(p2ann)_{{\rm Con}(A)}$ is equivalent to $(p2ann)_{{\rm Con}(A)/_{\scriptstyle \equiv _A}}$.
\item\label{echiv(iii)nobm3} If the equivalent conditions from (\ref{echiv(iii)nobm2}) are fulfilled, then the map from (\ref{echiv(iii)nobm1}) is a lattice isomorphism.\end{enumerate}\label{echiv(iii)nobm}\end{lemma}

\begin{proof} (\ref{echiv(iii)nobm1}) By Lemma \ref{annlambda}, this restriction of the direct image of $\lambda _A$ takes ${\rm Ann}({\rm Ann}(\theta ))$ to ${\rm Ann}({\rm Ann}(\lambda _A(\theta )))$ for all $\theta \in {\rm Con}(A)$, thus it is well defined and surjective. By Lemma \ref{annegal}, (\ref{annegal3}), for any $\theta ,\zeta \in {\rm Con}(A)$, $\lambda _A({\rm Ann}({\rm Ann}(\theta )))=\lambda _A({\rm Ann}({\rm Ann}(\zeta )))$ iff ${\rm Ann}({\rm Ann}(\theta ))={\rm Ann}({\rm Ann}(\zeta ))$, so this map is also injective. By Lemma \ref{annegal}, (\ref{annegal2}), for all $\theta ,\zeta \in {\rm Con}(A)$, ${\rm Ann}({\rm Ann}(\theta ))\subseteq {\rm Ann}({\rm Ann}(\zeta ))$ iff $\lambda _A({\rm Ann}({\rm Ann}(\theta )))\subseteq \lambda _A({\rm Ann}({\rm Ann}(\zeta )))$, hence this bijection and its inverse preserve order, so this map is an order isomorphism, which is a restriction of the lattice morphism $I\mapsto \lambda _A(I)$ from ${\rm Id}({\rm Con}(A))$ to ${\rm Id}({\rm Con}(A)/_{\textstyle \equiv _A})$.

\noindent (\ref{echiv(iii)nobm2}) and (\ref{echiv(iii)nobm3}) follow from (\ref{echiv(iii)nobm1}) in the same way in which properties (\ref{annbij2}), (\ref{annbij3}) and (\ref{annbij4}) from Lemma \ref{annbij} follow from (\ref{annbij1}). Another way to prove these facts is to notice that the equivalences in Lemma \ref{annbij}, (\ref{annbij2}), hold if we replace the congruences by sets of congruences, and also apply Lemmas \ref{semilatl} and \ref{semilata}.\end{proof}

\begin{proposition} For any nonzero cardinality $\kappa $, the properties $(3)_{\kappa ,{\rm Con}(A)}$ and $(3)_{\kappa ,{\rm Con}(A)/_{\textstyle \equiv _A}}$ are equivalent.\label{echiv(iii)}\end{proposition}

\begin{proof} By Lemma \ref{echiv(iii)nobm} and the fact that the map from Lemma \ref{echiv(iii)nobm}, (\ref{echiv(iii)nobm1}), composed with the map from $(3)_{\kappa ,{\rm Con}(A)}$ equals the map from $(3)_{\kappa ,{\rm Con}(A)/_{\textstyle \equiv _A}}$ composed with the the canonical surjective lattice morphism from ${\rm Con}(A)$ to ${\rm Con}(A)/_{\textstyle \equiv _A}$.\end{proof}

\begin{proposition} Properties $(iv)_{{\rm Con}(A)}$ and $(iv)_{{\rm Con}(A)/_{\scriptstyle \equiv _A}}$ are equivalent.\label{eq(iv)}\end{proposition}

\begin{proof} By Lemma \ref{annbij}, (\ref{annbij2}), and the surjectivity of the map $\lambda _A:{\rm Con}(A)\rightarrow {\rm Con}(A)/_{\textstyle \equiv _A}$.\end{proof}

\begin{remark} Since ${\rm Con}(A)/_{\textstyle \equiv _A}$ is a frame, ${\rm 2Ann}({\rm Con}(A)/_{\textstyle \equiv _A})\subseteq {\cal A}nn({\rm Con}(A)/_{\textstyle \equiv _A})={\rm PAnn}({\rm Con}(A)/_{\textstyle \equiv _A})$, which means that the second part of condition $(4)_{\kappa ,{\rm Con}(A)/_{\textstyle \equiv _A}}$ is fulfilled for any nonzero cardinality $\kappa $.\label{(iv)}\end{remark}

\begin{lemma} ${\cal A}nn({\rm Con}(A))={\rm PAnn}({\rm Con}(A))$.\label{panncg}\end{lemma}

\begin{proof} By Remark \ref{(iv)} and Lemma \ref{annlambda}, for any $\Omega \subseteq {\rm Con}(A))$, $\lambda _A({\rm Ann}(\Omega ))={\rm Ann}(\lambda _A(\Omega ))={\rm Ann}(\lambda _A(\theta ))=\lambda _A({\rm Ann}(\theta ))$ for some $\theta \in {\rm Con}(A))$, so that ${\rm Ann}(\Omega )={\rm Ann}(\theta )$ by Lemma \ref{annegal}, (\ref{annegal3}).\end{proof}

\begin{proposition} For any nonzero cardinality $\kappa $, the properties $(iv)_{{\rm Con}(A)}$, $(4)_{\kappa ,{\rm Con}(A)}$, $(4)_{<\infty ,{\rm Con}(A)}$ and $(4)_{{\rm Con}(A)}$ are equivalent.\label{echiv(iv)}\end{proposition}

\begin{proof} By Lemma \ref{panncg}, ${\rm 2Ann}({\rm Con}(A))\subseteq {\cal A}nn({\rm Con}(A))={\rm PAnn}({\rm Con}(A))$, which means that the second condition in $(4)_{\kappa ,{\rm Con}(A)}$ is fulfilled for any nonzero cardinality $\kappa $, so that conditions $(iv)_{{\rm Con}(A)}$, $(4)_{\kappa , {\rm Con}(A)}$, $(4)_{<\infty ,{\rm Con}(A)}$ and $(4)_{{\rm Con}(A)}$ are equivalent.\end{proof}

\begin{proposition} For any nonzero cardinality $\kappa $, the properties $(5)_{\kappa ,{\rm Con}(A)}$ and $(5)_{\kappa ,{\rm Con}(A)/_{\textstyle \equiv _A}}$ are equivalent.\label{echiv(v)}\end{proposition}

\begin{proof} Assume that $(5)_{\kappa ,{\rm Con}(A)}$ is fulfilled and let $U\subseteq {\rm Con}(A)/_{\textstyle \equiv _A}$ with $|U|\leq \kappa $. For each $u\in U$, there exists an $\omega _u\in {\rm Con}(A)$ such that $\lambda _A(\omega _u)=u$. Let $\Omega =\{\omega _u\ |\ u\in U\}\subseteq {\rm Con}(A)$. Then $|\Omega |=|U|\leq \kappa $ and $\lambda _A(\Omega )=U$, thus ${\rm Ann}(\Omega )\vee {\rm Ann}({\rm Ann}(\Omega ))={\rm Con}(A)$ and hence, by Lemma \ref{annlambda}, ${\rm Ann}(U)\vee {\rm Ann}({\rm Ann}(U))={\rm Ann}(\lambda _A(\Omega ))\vee {\rm Ann}({\rm Ann}(\lambda _A(\Omega )))=\lambda _A({\rm Ann}(\Omega ))\vee \lambda _A({\rm Ann}({\rm Ann}(\Omega )))=\lambda _A({\rm Ann}(\Omega )\vee {\rm Ann}({\rm Ann}(\Omega )))=\lambda _A({\rm Con}(A))=$\linebreak ${\rm Con}(A)/_{\textstyle \equiv _A}$, so the direct implication holds.

Now assume that $(5)_{\kappa ,{\rm Con}(A)/_{\textstyle \equiv _A}}$ is fulfilled and let $\Omega \subseteq {\rm Con}(A)$ with $|\Omega |\leq \kappa $. Then $\lambda _A(\Omega )\subseteq {\rm Con}(A)/_{\textstyle \equiv _A}$ and $|\lambda _A(\Omega )|\leq |\Omega |\leq \kappa $, hence ${\rm Ann}(\lambda _A(\Omega ))\vee {\rm Ann}({\rm Ann}(\lambda _A(\Omega )))={\rm Con}(A)/_{\textstyle \equiv _A}=\lambda _A({\rm Con}(A))$, so that $\lambda _A(\nabla _A)\in {\rm Ann}(\lambda _A(\Omega ))\vee {\rm Ann}({\rm Ann}(\lambda _A(\Omega )))$, which means that $\lambda _A(\nabla _A)\leq \lambda _A(\theta )\vee \lambda _A(\zeta )$ for some $\theta ,\zeta \in {\rm Con}(A)$ such that $\lambda _A(\theta )\in {\rm Ann}(\lambda _A(\Omega ))=\lambda _A({\rm Ann}(\Omega ))$ and $\lambda _A(\zeta )\in {\rm Ann}({\rm Ann}(\lambda _A(\Omega )))=\lambda _A({\rm Ann}({\rm Ann}(\Omega )))$, so that $\theta \in {\rm Ann}(\Omega )$ and $\zeta \in {\rm Ann}({\rm Ann}(\Omega ))$, by Lemma \ref{annlambda} and Lemma \ref{annegal}, (\ref{annegal1}). Hence $\lambda _A(\nabla _A)=\lambda _A(\theta )\vee \lambda _A(\zeta )=\lambda _A(\theta \vee \zeta )$, thus $\nabla _A=\theta \vee \zeta \in {\rm Ann}(\Omega )\vee {\rm Ann}({\rm Ann}(\Omega ))$, hence ${\rm Ann}(\Omega )\vee {\rm Ann}({\rm Ann}(\Omega ))={\rm Con}(A)$, therefore the converse implication holds, as well.\end{proof}

\begin{theorem} For any nonzero cardinality $\kappa $ and any $h,i,j\in \overline{1,5}$, conditions $(iv)_{{\rm Con}(A)}$, $(h)_{\kappa ,{\rm Con}(A)}$,\linebreak $(i)_{<\infty ,{\rm Con}(A)}$ and $(j)_{{\rm Con}(A)}$ are equivalent.\label{mydavey}\end{theorem}

\begin{proof} By Corollary \ref{cordav} and Propositions \ref{echiv(i)}, \ref{echiv(ii)}, \ref{echiv(iii)}, \ref{echiv(iv)} and \ref{echiv(v)}.\end{proof}

\begin{openproblem} Determine what kinds of bounded modular lattices are congruence lattices of semiprime algebras from semi--degenerate congruence--modular varieties. It will follow that the equivalences in the last statement in Corollary \ref{moredavey} hold for all those kinds of bounded modular lattices.\end{openproblem}

\begin{corollary} Let $\kappa $ be a nonzero cardinality. Then: the equivalent conditions $(iv)_{{\rm Con}(A)},(1)_{\kappa ,{\rm Con}(A)},\ldots ,$\linebreak $(5)_{\kappa ,{\rm Con}(A)},(1)_{<\infty ,{\rm Con}(A)},\ldots ,(5)_{<\infty ,{\rm Con}(A)},(1)_{{\rm Con}(A)},\ldots ,(5)_{{\rm Con}(A)}$ are fulfilled iff the equivalent conditions\linebreak $(iv)_{{\rm Con}(A)/_{\scriptstyle \equiv _A}},\ (1)_{\kappa ,{\rm Con}(A)/_{\scriptstyle \equiv _A}},\ldots ,(5)_{\kappa ,{\rm Con}(A)/_{\scriptstyle \equiv _A}},\ (1)_{<\infty ,{\rm Con}(A)/_{\scriptstyle \equiv _A}},\ldots ,(5)_{<\infty ,{\rm Con}(A)/_{\scriptstyle \equiv _A}},\ (1)_{{\rm Con}(A)/_{\scriptstyle \equiv _A}},\ldots ,$\linebreak $(5)_{{\rm Con}(A)/_{\scriptstyle \equiv _A}}$ are fulfilled.\label{cormydavey}\end{corollary}

\begin{proof} By any of the Propositions \ref{eq(iv)}, \ref{echiv(i)}, \ref{echiv(ii)}, \ref{echiv(iii)}, \ref{echiv(iv)} and \ref{echiv(v)}, along with Corollary \ref{cordav} and Theorem \ref{mydavey}.\end{proof}

In \cite{retic}, we have constructed the {\em reticulation} of $A$, ${\cal L}(A)$, which, by definition, is a bounded distributive lattice whose prime spectrum of ideals (or filters, but our construction in \cite{retic} fulfills this property for ideals) is homeomorphic to the prime spectrum of congruences of $A$, with respect to the Stone topologies. It is well known that, if two bounded distributive lattices have homeomorphic prime spectra of ideals, then they are isomorphic lattices, therefore the reticulation of $A$ is unique up to a lattice isomorphism (or dual isomorphism, if we also consider the variant of the reticulation with the property above for filters). This is our construction of the reticulation of $A$ from \cite{retic}: ${\cal L}(A)={\cal K}(A)/_{\textstyle \equiv _A}$, where ${\cal K}(A)$ is the set of the compact elements of the lattice ${\rm Con}(A)$, thus ${\cal L}(A)={\rm Con}(A)/_{\textstyle \equiv _A}$ if the lattice ${\rm Con}(A)$ is compact, in particular if ${\rm Con}(A)$ is a finite lattice, in particular if $A$ is finite. Following \cite{retic}, I am denoting the restriction $\lambda _A\mid _{{\cal K}(A)}:{\cal K}(A)\rightarrow {\cal L}(A)$ of the canonical surjective lattice morphism $\lambda _A:{\rm Con}(A)\rightarrow {\rm Con}(A)/_{\textstyle \equiv _A}$ by $\lambda _A$, as well. ${\cal L}(A)$ is a bounded sublattice of ${\rm Con}(A)/_{\textstyle \equiv _A}$, thus a bounded distributive lattice, hence, from Corollaries \ref{moredavey} and \ref{cordav} and the obvious fact that a complete sublattice of a frame is a frame, we obtain:

\begin{corollary}\begin{itemize}
\item Conditions $(1)_{{\cal L}(A)},\ldots ,(5)_{{\cal L}(A)}$ are equivalent. For any nonzero cardinality $\kappa $, $(1)_{\kappa ,{\cal L}(A)},$\linebreak $\ldots ,(5)_{\kappa ,{\cal L}(A)}$ are equivalent. For any finite nonzero cardinality $\kappa $, $(1)_{\kappa ,{\cal L}(A)},\ldots ,(5)_{\kappa ,{\cal L}(A)},(1)_{<\infty ,{\cal L}(A)},\ldots ,$\linebreak $(5)_{<\infty ,{\cal L}(A)}$ are equivalent.
\item If ${\cal L}(A)$ is a frame, in particular if ${\cal L}(A)$ is a complete sublattice of ${\rm Con}(A)/_{\textstyle \equiv _A}$, in particular if ${\cal L}(A)={\rm Con}(A)/_{\textstyle \equiv _A}$, in particular if ${\rm Con}(A)$ is a compact lattice, then, for any nonzero cardinality $\kappa $ and any $h,i,j\in \overline{1,5}$, $(iv)_{{\cal L}(A)}$, $(h)_{\kappa ,{\cal L}(A)}$, $(i)_{<\infty ,{\cal L}(A)}$ and $(j)_{{\cal L}(A)}$ are equivalent, so, in particular, $(1)_{{\cal L}(A)}$ is equivalent to $(2)_{<\infty ,{\cal L}(A)}$, which means that: ${\cal L}(A)$ is a Stone lattice iff ${\cal L}(A)$ is a strongly Stone lattice.\end{itemize}\label{cormoredavey}\end{corollary}

\begin{corollary} If ${\cal L}(A)={\rm Con}(A)/_{\textstyle \equiv _A}$, in particular if ${\rm Con}(A)$ is a compact lattice, in particular if ${\rm Con}(A)$ is finite, in particular if $A$ is finite, then: ${\rm Con}(A)$ is Stone iff ${\rm Con}(A)$ is strongly Stone iff ${\cal L}(A)$ is Stone iff ${\cal L}(A)$ is strongly Stone.\label{stoneretic}\end{corollary}

\begin{remark} Let $M$ be a bounded sublattice of $L$ and $U,V\subseteq L$. Then it is straightforward that ${\rm Ann}_L(U)\cap M\subseteq {\rm Ann}_M(U\cap M)$, $(U\cap M]_L\cap M=(U\cap M]_M$ and, if $L$ is distributive, then $(U]_L\cap (V]_L=(U\cap V]_L$.\end{remark}

\begin{lemma} If $L$ is a bounded distributive lattice, $M$ is a bounded sublattice of $L$ and $U\subseteq M$, such that ${\rm Ann}_L(U)\vee {\rm Ann}_L({\rm Ann}_L(U))=L$, then ${\rm Ann}_M(U)\vee {\rm Ann}_M({\rm Ann}_M(U))=M$.\label{sublat}\end{lemma}

\begin{proof} If $U$ is as in the hypothesis, then ${\rm Ann}_M(U)\vee {\rm Ann}_M({\rm Ann}_M(U))\supseteq ({\rm Ann}_L(U)\cap M)\vee ({\rm Ann}_M({\rm Ann}_L(U)\cap M)\supseteq ({\rm Ann}_L(U)\cap M)\vee ({\rm Ann}_L({\rm Ann}_L(U))\cap M)\supseteq ({\rm Ann}_L(U)\cap M)\cup ({\rm Ann}_L({\rm Ann}_L(U))\cap M)=({\rm Ann}_L(U)\cup {\rm Ann}_L({\rm Ann}_L(U))\cap M$, thus ${\rm Ann}_M(U)\vee {\rm Ann}_M({\rm Ann}_M(U))\supseteq ({\rm Ann}_L(U)\cup {\rm Ann}_L({\rm Ann}_L(U)]_M\cap (M]_M=({\rm Ann}_L(U)\cup {\rm Ann}_L({\rm Ann}_L(U))]_L\cap M\cap M=({\rm Ann}_L(U)\vee {\rm Ann}_L({\rm Ann}_L(U)))\cap M=L\cap M=M$.\end{proof}

\begin{lemma} If $L$ is a bounded distributive lattice and $M$ is a bounded sublattice of $L$, then, for any nonzero cardinality $\kappa $, $(5)_{\kappa ,L}$ implies $(5)_{\kappa ,M}$.\label{(v)sublat}\end{lemma}

\begin{proof} Assume that $(5)_{\kappa ,L}$ is fulfilled, and let $U\subseteq M\subseteq L$ with $|U|\leq \kappa $, so that ${\rm Ann}_L(U)\vee {\rm Ann}_L({\rm Ann}_L(U))=L$, hence ${\rm Ann}_M(U)\vee {\rm Ann}_M({\rm Ann}_M(U))=M$ by Lemma \ref{sublat}.\end{proof}

\begin{proposition} If $L$ is a bounded distributive lattice, $M$ is a bounded sublattice of $L$, $\kappa $ is a nonzero cardinality and the equivalent conditions $(1)_{\kappa ,L},\ldots ,(5)_{\kappa ,L}$ are fulfilled, then the equivalent conditions $(1)_{\kappa ,M},\ldots ,(5)_{\kappa ,M}$ are fulfilled.\label{echivsublat}\end{proposition}

\begin{proof} By Corollary \ref{moredavey} and Lemma \ref{(v)sublat}.\end{proof}

\begin{corollary} For any nonzero cardinality $\kappa $, if the equivalent conditions $(iv)_{{\rm Con}(A)},\!(1)_{\kappa ,{\rm Con}(A)},\!\ldots \!,(5)_{\kappa ,{\rm Con}(A)},$\linebreak $(1)_{<\infty ,{\rm Con}(A)},\ldots ,(5)_{<\infty ,{\rm Con}(A)},(1)_{{\rm Con}(A)},\ldots ,(5)_{{\rm Con}(A)}$ are fulfilled, then conditions $(iv)_{{\cal L}(A)},(1)_{\kappa ,{\cal L}(A)},\ldots ,$\linebreak $(5)_{\kappa ,{\cal L}(A)},(1)_{<\infty ,{\cal L}(A)},\ldots ,(5)_{<\infty ,{\cal L}(A)},(1)_{{\cal L}(A)},\ldots ,(5)_{{\cal L}(A)}$ (not necessarily equivalent) are fulfilled.\label{corechivsublat}\end{corollary}

\begin{remark} Obviously, if $L$ and $M$ are isomorphic bounded lattices, then: $(iv)_L$ is equivalent to $(iv)_M$, for any $i\in \overline{1,5}$, $(i)_L$ is equivalent to $(i)_M$, $(i)_{<\infty ,L}$ is equivalent to $(i)_{<\infty ,M}$ and, for any nonzero cardinality $\kappa $, $(i)_{\kappa ,L}$ is equivalent to $(i)_{\kappa ,M}$.\end{remark}

\begin{remark} If $L$ and $M$ are bounded distributive lattices, then it is immediate that: ${\cal B}(L\times M)={\cal B}(L)\times {\cal B}(M)$ and, for all $a\in L$ and all $b\in M$, ${\rm Ann}_{L\times M}((a,b))=\{(x,y)\ |\ x\in {\rm Ann}_L(a),y\in {\rm Ann}_M(b)\}={\rm Ann}_L(a)\times {\rm Ann}_M(b)$; since, for any $x\in L$ and $y\in M$, $(x,y)\leq (a,b)$ means that $x\leq a$ and $y\leq b$, it follows that $((a,b)]_{L\times M}=\{(x,y)\ |\ x\in (a]_L,y\in (b]_M\}=(a]_L\times (b]_M$. And, for any set $I$ and any families $(a_i)_{i\in I}\subseteq L$ and $(b_i)_{i\in I}\subseteq M$, there exists $\displaystyle \bigvee _{i\in I}(a_i,b_i)$ in $L\times M$ iff there exist $\displaystyle \bigvee _{i\in I}a_i$ in $L$ and $\displaystyle \bigvee _{i\in I}b_i$ in $M$, and, if these joins exist, then $\displaystyle \bigvee _{i\in I}(a_i,b_i)=(\bigvee _{i\in I}a_i,\bigvee _{i\in I}b_i)$; the same goes for arbitrary meets. From this, it is easy to obtain that, for any nonzero cardinality $\kappa $, $(2)_{\kappa ,L\times M}$ is fulfilled iff both $(2)_{\kappa ,L}$ and $(2)_{\kappa ,M}$ are fulfilled, and hence, by Corollary \ref{moredavey}:\end{remark}

\begin{proposition} For any nonzero cardinality $\kappa $: condition $(iv)_{L\times M}$, respectively the equivalent conditions $(1)_{\kappa ,L\times M},\ldots ,(5)_{\kappa ,L\times M}$, respectively $(1)_{<\infty ,L\times M},\ldots ,(5)_{<\infty ,L\times M}$, respectively $(1)_{L\times M},\ldots ,(5)_{L\times M}$, are fulfilled iff condition $(iv)_L$, respectively the equivalent conditions $(1)_{\kappa ,L},\ldots ,(5)_{\kappa ,L}$, respectively $(1)_{<\infty ,L},\ldots ,(5)_{<\infty ,L}$, respectively $(1)_L,\ldots ,(5)_L$, as well as condition $(iv)_M$, respectively the equivalent conditions $(1)_{\kappa ,M},\ldots ,(5)_{\kappa ,M}$, respectively $(1)_{<\infty ,M},\ldots ,(5)_{<\infty ,M}$, respectively $(1)_M,\ldots ,(5)_M$, are fulfilled.\label{echivprod}\end{proposition}

Throughout the rest of this section, $B$ shall be a semiprime algebra from ${\cal C}$. Then:

\begin{remark} ${\rm Con}(A\times B)$ is isomorphic to ${\rm Con}(A)\times {\rm Con}(B)$ and, as we have proven in \cite{retic}, $A\times B$ is semiprime, ${\rm Con}(A\times B)/_{\textstyle \equiv _{A\times B}}$ is isomorphic to ${\rm Con}(A)/_{\textstyle \equiv _A}\times {\rm Con}(B)/_{\textstyle \equiv _B}$ and ${\cal L}(A\times B)$ is isomorphic to ${\cal L}(A)\times {\cal L}(B)$.\label{prodalg}\end{remark}

\begin{corollary}\begin{itemize}
\item For any nonzero cardinality $\kappa $: the equivalent conditions $(iv)_{{\rm Con}(A\times B)/_{\textstyle \equiv _{A\times B}}},$\linebreak $(1)_{\kappa ,{\rm Con}(A\times B)/_{\textstyle \equiv _{A\times B}}},\ \ldots \ ,(5)_{\kappa ,{\rm Con}(A\times B)/_{\textstyle \equiv _{A\times B}}},\ (1)_{<\infty ,{\rm Con}(A\times B)/_{\textstyle \equiv _{A\times B}}},\ \ldots \ ,(5)_{<\infty ,{\rm Con}(A\times B)/_{\textstyle \equiv _{A\times B}}},$\linebreak $(1)_{{\rm Con}(A\times B)/_{\textstyle \equiv _{A\times B}}},\ldots ,(5)_{{\rm Con}(A\times B)/_{\textstyle \equiv _{A\times B}}}$ are fulfilled iff the equivalent conditions $(iv)_{{\rm Con}(A)/_{\scriptstyle \equiv _A}},$\linebreak $(1)_{\kappa ,{\rm Con}(A)/_{\scriptstyle \equiv _A}},\ldots ,(5)_{\kappa ,{\rm Con}(A)/_{\scriptstyle \equiv _A}},(1)_{<\infty ,{\rm Con}(A)/_{\scriptstyle \equiv _A}},\ldots ,(5)_{<\infty ,{\rm Con}(A)/_{\scriptstyle \equiv _A}},(1)_{{\rm Con}(A)/_{\scriptstyle \equiv _A}},\ldots ,(5)_{{\rm Con}(A)/_{\scriptstyle \equiv _A}}$, as well as the equivalent conditions $(iv)_{{\rm Con}(B)/_{\scriptstyle \equiv _B}},(1)_{\kappa ,{\rm Con}(B)/_{\scriptstyle \equiv _B}},\ldots ,(5)_{\kappa ,{\rm Con}(B)/_{\scriptstyle \equiv _B}},(1)_{<\infty ,{\rm Con}(B)/_{\scriptstyle \equiv _B}},\ldots ,$\linebreak $(5)_{<\infty ,{\rm Con}(B)/_{\scriptstyle \equiv _B}},(1)_{{\rm Con}(B)/_{\scriptstyle \equiv _B}},\ldots ,(5)_{{\rm Con}(B)/_{\scriptstyle \equiv _B}}$, are fulfilled.
\item For any nonzero cardinality $\kappa $: condition $(iv)_{{\cal L}(A\times B)}$, respectively the equivalent conditions $(1)_{\kappa ,{\cal L}(A\times B)},\ldots ,$\linebreak $(5)_{\kappa ,{\cal L}(A\times B)}$, respectively $(1)_{<\infty ,{\cal L}(A\times B)},\ldots ,(5)_{<\infty ,{\cal L}(A\times B)}$, respectively $(1)_{{\cal L}(A\times B)},\ldots ,(5)_{{\cal L}(A\times B)}$, are fulfilled iff condition $(iv)_{{\cal L}(A)}$, respectively the equivalent conditions $(1)_{\kappa ,{\cal L}(A)},\ldots ,(5)_{\kappa ,{\cal L}(A)}$, respectively\linebreak  $(1)_{<\infty ,{\cal L}(A)},\ldots ,(5)_{<\infty ,{\cal L}(A)}$, respectively $(1)_{{\cal L}(A)},\ldots ,(5)_{{\cal L}(A)}$, as well as condition $(iv)_{{\cal L}(B)}$, respectively the equivalent conditions $(1)_{\kappa ,{\cal L}(B)},\ldots ,(5)_{\kappa ,{\cal L}(B)}$, respectively $(1)_{<\infty ,{\cal L}(B)},\ldots ,(5)_{<\infty ,{\cal L}(B)}$, res\-pec\-ti\-vely $(1)_{{\cal L}(B)},\linebreak \ldots ,(5)_{{\cal L}(B)}$, are fulfilled.\end{itemize}\label{echivprodalg}\end{corollary}

\begin{proof} By Remark \ref{prodalg}, Proposition \ref{echivprod} and Corollary \ref{moredavey}.\end{proof}

\begin{corollary} For any nonzero cardinality $\kappa $: the equivalent conditions $(iv)_{{\rm Con}(A\times B)},(1)_{\kappa ,{\rm Con}(A\times B)},\ldots ,$\linebreak $(5)_{\kappa ,{\rm Con}(A\times B)},(1)_{<\infty ,{\rm Con}(A\times B)},\ldots ,(5)_{<\infty ,{\rm Con}(A\times B)},(1)_{{\rm Con}(A\times B)},\ldots ,(5)_{{\rm Con}(A\times B)}$ are fulfilled iff the equivalent condition $(iv)_{{\rm Con}(A)},(1)_{\kappa ,{\rm Con}(A)},\ldots ,(5)_{\kappa ,{\rm Con}(A)},(1)_{<\infty ,{\rm Con}(A)},\ldots ,(5)_{<\infty ,{\rm Con}(A)},(1)_{{\rm Con}(A)},\ldots ,(5)_{{\rm Con}(A)}$, as well as the equivalent conditions $(iv)_{{\rm Con}(B)},(1)_{\kappa ,{\rm Con}(B)},\ldots ,(5)_{\kappa ,{\rm Con}(B)},(1)_{<\infty ,{\rm Con}(B)},\ldots ,(5)_{<\infty ,{\rm Con}(B)},(1)_{{\rm Con}(B)},\ldots ,$\linebreak $(5)_{{\rm Con}(B)}$, are fulfilled.\label{corprodalg}\end{corollary}

\begin{proof} By Corollaries \ref{echivprodalg} and \ref{cormydavey}.\end{proof}

\section{Transferring Davey`s Theorem to Commutative Unitary Rings}
\label{resultselem}

Let $(T,\vee ,\wedge ,\odot ,\rightarrow ,0,1)$ be a {\em residuated lattice}, which means that $(T,\vee ,\wedge ,0,1)$ is a bounded lattice, that we shall denote by $S$, $(T,\odot ,1)$ is a commutative monoid and $\rightarrow $ is a binary operation on $T$ which fulfills the {\em law of residuation}: for all $a,b,c\in T$, $a\leq b\rightarrow c$ iff $a\odot b\leq c$. Let us denote by $S^{\prime }$ the dual of the bounded lattice $S$. See more about residuated lattices in \cite{gal}, \cite{ior}, \cite{pic}.

Residuated lattices form a semi--degenerate congruence--distributive variety, hence they are semiprime and thus their congruence lattices fulfill Theorem \ref{davey} and even Theorem \ref{mydavey}. But they also fulfill a theorem of this form for elements, which can be expressed in the following way, since we notice that the bounded lattice of the filters of $T$ is a bounded sublattice of that of the filters of $S$, for each $e\in {\cal B}(S)$, the filter of $T$ generated by $e$ coincides to the filter of $S$ generated by $e$ and the co--annihilators in $T$ coincide to the co--annihilators in $S$:

\begin{theorem}{\rm \cite[Theorem $5.2.6$]{eu},\cite[Theorem $3.13$]{eu7}} Conditions $(1)_{m,S^{\prime }}$, $(2)_{m,S^{\prime }}$, $(3)_{m,S^{\prime }}$, $(4)_{m,S^{\prime }}$ and $(5)_{m,S^{\prime }}$ are equivalent.\label{elemreslat}\end{theorem}

In \cite{eu,eu7}, I have proven Theorem \ref{elemreslat} by transferring the dual of Theorem \ref{davey} from bounded distributive lattices to residuated lattices through the reticulation functor for residuated lattices. With the notation for the reticulation from Section \ref{resultscg}, the construction from \cite{retic} identifies ${\cal L}(T)$, up to a lattice isomorphism, as the bounded lattice ${\rm PFilt}(T)$ of the principal filters of $T$, whose dual is a frame if $T$ is complete. Since, for any nonzero cardinality $\kappa $ and any $i\in \overline{1,5}$, $(i)_{\kappa ,S^{\prime }}$ is equivalent to $(i)_{\kappa ,D}$, where $D$ is the dual of ${\rm PFilt}(T)$, according to \cite{eu,eu7} and the above, it follows that:

\begin{theorem} If $T$ is a complete residuated lattice and $S^{\prime }$ is the dual of the underlying bounded lattice of $T$, then, for any nonzero cardinality $\kappa $ and any $h,i,j\in \overline{1,5}$, conditions $(iv)_{S^{\prime }}$, $(h)_{\kappa ,S^{\prime }}$, $(i)_{<\infty ,S^{\prime }}$ and $(j)_{S^{\prime }}$ are equivalent; in particular, $T$ is co--Stone iff $T$ is strongly co--Stone.\label{moreelemreslat}\end{theorem}

\begin{remark} If the commutator of $A$ is associative, as it is, for example, in any commutative unitary ring, then ${\rm Con}(A)$ is a residuated lattice, in which $[\cdot ,\cdot ]_A$ is the multiplication, according to \cite{retic} (see also \cite{cze2}), thus, from the above and the fact that ${\rm Con}(A)$ is a complete lattice, it follows that both the lattice ${\rm Con}(A)$ and its dual fulfill the equivalences in Theorem \ref{mydavey}. Note that the commutator is not always associative \cite{ggiv}.\end{remark}

Commutative unitary rings form a semi--degenerate congruence--modular equational class, thus their congruence lattices fulfill Theorem \ref{mydavey}. Let us see that, similarly to what happens in (complete) residuated lattices, they also fulfill an analogue of Theorem \ref{mydavey} for elements instead of congruences.

Commutative unitary rings form a semi--degenerate congruence--modular variety, thus their congruence lattices fulfill Theorem \ref{mydavey}. Let us see that, like residuated lattices, commutative unitary rings fulfill Davey`s Theorem for elements, too.

Let $(R,+,\cdot ,0,1)$ be a commutative unitary ring, $({\rm Id}(R),\vee =+,\cap ,\{0\},R)$ be the bounded modular lattice of the ideals of $R$, ${\rm Spec}_{\rm Id}(R)$ the set of the prime ideals of $R$ and $\iota\gamma _R:{\rm Id}(R)\rightarrow {\rm Con}(R)$ the canonical lattice isomorphism: for all $I\in {\rm Id}(R)$, $\iota\gamma _R(I)=\{(x,y)\in I^2\ |\ x-y\in I\}$. Note that, since $\iota\gamma _R$ is an order isomorphism, it preserves arbitrary intersections. Recall that ${\rm Spec}_{\rm Id}(R)=\{P\in {\rm Id}(R)\setminus \{R\}\ |\ (\forall \, I,J\in {\rm Id}(R))\, (I\cdot J\subseteq P\Rightarrow I\subseteq P\mbox{ or }J\subseteq P)\}$ and $[\iota\gamma _R(I),\iota\gamma _R(J)]_R=\iota\gamma _R(I\cdot J)$ for all $I,J\in {\rm Id}(R)$, from which it is easy to deduce that $\iota\gamma _R({\rm Spec}_{\rm Id}(R))={\rm Spec}(R)$. For every $U\subseteq R$, $\langle U\rangle _R$ shall be the ideal of $R$ generated by $U$, so, for each $x\in R$, $\langle \{x\}\rangle _R=xR$. Let ${\rm PId}(R)$ and ${\rm FGId}(R)$ be the set of the principal ideals of $R$ and that of the finitely generated ideals of $R$, respectively. It is straightforward that, for all $x,a,b\in R$, $\iota\gamma _R(xR)=Cg_R(x,0)$ and $Cg_R(a,b)=Cg_R(a-b,0)$, hence $\iota\gamma _R({\rm PId}(R))={\rm PCon}(R)$ and thus $\iota\gamma _R({\rm FGId}(R))={\rm K}(R)$. Recall that, if we denote by $E(R)$ the set of the idempotents of $R$, then $(E(R),\vee ,\wedge =\cdot ,\neg \, ,0,1)$ is a Boolean algebra, where, for every $e,f\in E(R)$, $\neg \, e=1-e$ and $e\vee f=\neg \, (\neg \, e\wedge \neg \, f)=1-(1-e)\cdot (1-f)$.

Let ${\cal L}(R)=({\cal K}(R)/_{\textstyle \equiv _R},\vee ,\wedge ,{\bf 0}=\lambda _R(\Delta _R)=\Delta _R/_{\textstyle \equiv _R},{\bf 1}=\lambda _R(\nabla _R)=\nabla _R/_{\textstyle \equiv _R})$ be the reticulation of $R$, as constructed in \cite{retic} (see the notations in Section \ref{resultscg}). Let $R^*$ and $\mu _R:R\rightarrow R^*$ be the reticulation of $R$ and the reticulation function, respectively, as constructed in \cite{bell,bell2} (see also \cite{joy,sim}): if we denote, for each $I\in {\rm Id}(R)$, by $\sqrt{I}=\bigcap \{P\in {\rm Spec}_{\rm Id}(R)\ |\ I\subseteq P\}$ the {\em radical} of $I$, and by $\sim _R=\{(I,J)\in ({\rm Id}(R))^2\ |\ \sqrt{I}=\sqrt{J}\}$, then $\sim _R$ is a congruence of the lattice ${\rm Id}(R)$, so ${\rm Id}(R)/_{\textstyle \sim _R}$ is a bounded lattice and the canonical surjection $\nu _R:{\rm Id}(R)\rightarrow {\rm Id}(R)/_{\textstyle \sim _R}$ is a lattice morphism that fulfills: $\nu _R(I\cdot J)=\nu _R(I\cap J)$ for all $I,J\in {\rm Id}(R)$, therefore the bounded lattice $R^*=({\rm Id}(R)/_{\textstyle \sim _R},\vee ,\wedge ,\perp =\nu _R(\{0\}))=\{0\}/_{\textstyle \sim _R},\top =\nu _R(R)=R/_{\textstyle \sim _R})$ is distributive, since multiplication is distributive w.r.t. the join in ${\rm Id}(R)$. As shown in \cite{bell}, $R^*$ has the prime spectrum of ideals homeomorphic to the prime spectrum of congruences of $R$, w.r.t. the Stone topologies, hence $R^*$ is a reticulation of $R$. The {\em reticulation function} $\mu _R:R\rightarrow R^*$ is defined, in \cite{bell}, by: $\mu _R(x)=\nu _R(xR)=xR/_{\textstyle \sim _R}$ for all $x\in R$.

For any $U\subseteq R$, we denote the {\em annihilator of $U$} by ${\rm Ann}_R(U)$; so ${\rm Ann}_R(U)=\{x\in R\ |\ (\forall \, u\in U)\, (u\cdot x=0)\}\in {\rm Id}(R)$. For any $a\in R$, we also denote ${\rm Ann}_R(a)={\rm Ann}_R(\{a\})=\{x\in R\ |\ u\cdot a=0\}$. Let us denote by ${\cal A}nn(R)=\{{\rm Ann}_R(U)\ |\ U\subseteq R\}$ and by ${\rm 2Ann}(R)=\{{\rm Ann}_R({\rm Ann}_R(U))\ |\ U\subseteq R\}$.

\begin{remark} It is well known and straightforward that ${\cal A}nn(R)\subseteq {\rm Id}(R)$.\end{remark}

\begin{lemma}\begin{itemize}
\item For any $U\subseteq R$, ${\rm Ann}_R(U)={\rm Ann}_R(\langle U\rangle _R)$.
\item For any $V\subseteq {\rm Id}(R)$, $\displaystyle \bigcap _{I\in V}{\rm Ann}_R(I)={\rm Ann}_R(\bigvee _{I\in V}I)$.
\end{itemize}\label{proprinel}\end{lemma}

\begin{proof} Clearly, for any $U\subseteq V\subseteq R$, we have ${\rm Ann}_R(V)\subseteq {\rm Ann}_R(U)$, hence ${\rm Ann}_R(\langle U\rangle _R)\subseteq {\rm Ann}_R(U)$. But the converse inclusion holds, as well, since, given any $a\in \langle U\rangle _R$ and any $x\in {\rm Ann}_R(U)$, we have $a=a_1\cdot u_1+\ldots +a_n\cdot u_n$ for some $n\in \N ^*$, $a_1,\ldots ,a_n\in R$ and $u_1,\ldots ,u_n\in U$, so that $x\cdot u_1=\ldots =x\cdot u_n=0$, therefore $x\cdot u=0$, so $x\in {\rm Ann}_R(\langle U\rangle _R)$. Thus ${\rm Ann}_R(U)={\rm Ann}_R(\langle U\rangle _R)$. Hence $\displaystyle \bigcap _{I\in V}{\rm Ann}_R(I)=\bigcap _{I\in V}\bigcap _{a\in I}{\rm Ann}_R(a)=\bigcap _{a\in \bigcup _{I\in V}I}{\rm Ann}_R(a)={\rm Ann}_R(\bigcup _{I\in V}I)={\rm Ann}_R(\langle \bigcup _{I\in V}I\rangle _R)={\rm Ann}_R(\bigvee _{I\in V}I)$.\end{proof}

Regarding the results from \cite{bell} I am using, note that, since $R$ is commutative, it follows that $R$ is quasicommutative, thus, by \cite[Theorem $3$]{bell}, $R$ fulfills condition $(*)$ from \cite{bell}. It also follows that:

\begin{lemma}{\rm \cite[Lemma, p. 1861]{bell}} For all $I\in {\rm Id}(R)$, there exists a $K\in {\rm FGId}(R)$ such that $K\subseteq I$ and $\sqrt{K}=\sqrt{I}$.\label{hereswhy}\end{lemma}

\begin{proposition}\begin{enumerate}
\item\label{hereshow1} $R^*={\rm Id}(R)/
_{\textstyle \sim _R}={\rm FGId}(R)/
_{\textstyle \sim _R}$;
\item\label{hereshow0} $\varphi _R:R^*\rightarrow {\rm Con}(R)/_{\textstyle \equiv _R}$, for all $I\in {\rm Id}(R)$, $\varphi _R(I/_{\textstyle \sim _R})=\iota\gamma _R(I)/_{\textstyle \equiv _R}$, is a lattice isomorphism;
\item\label{hereshow2} ${\rm Con}(R)/
_{\textstyle \equiv _R}={\cal K}(R)/
_{\textstyle \equiv _R}={\cal L}(R)$.\end{enumerate}\label{hereshow}\end{proposition}

\begin{proof} (\ref{hereshow1}) Lemma \ref{hereswhy} says that, for all $I\in {\rm Id}(R)$, there exists a $K\in {\rm FGId}(R)$ such that $K\subseteq I$ and $K/_{\textstyle \sim _R}=I/_{\textstyle \sim _R}$, therefore $R^*={\rm Id}(R)/
_{\textstyle \sim _R}={\rm FGId}(R)/
_{\textstyle \sim _R}$.

\noindent (\ref{hereshow0}) The map $\varphi _R$ is defined by: $\varphi _R(\nu _R(I))=\lambda _R(\iota\gamma _R(I)))$ for all $I\in {\rm Id}(R)$, hence it makes the following diagram commutative, where the second equality in the bottom row will follow shortly:\vspace*{-10pt}

\begin{center}
\begin{picture}(325,60)(0,0)
\put(20,40){${\rm FGId}(R)\subseteq {\rm Id}(R)$}
\put(0,5){${\rm FGId}(R)/
_{\textstyle \sim _R}={\rm Id}(R)/
_{\textstyle \sim _R}=R^*$}
\put(86,25){$\nu _R$}
\put(29,25){$\nu _R\mid _{{\rm FGId}(R)}$}
\put(84,36){\vector(0,-1){21}}
\put(27,38){\vector(0,-1){23}}
\put(98,43){\vector(1,0){100}}
\put(213,38){\vector(0,-1){23}}
\put(257,36){\vector(0,-1){21}}
\put(259,25){$\lambda _R$}
\put(215,25){$\lambda _R$}
\put(137,46){$\iota\gamma _R$}
\put(152,11){$\varphi _R$}
\put(142,8){\vector(1,0){36}}
\put(200,40){${\rm Con}(R)\supseteq {\cal K}(R)$}
\put(180,5){${\rm Con}(R)/
_{\textstyle \equiv _R}={\cal K}(R)/
_{\textstyle \equiv _R}={\cal L}(R)$}
\end{picture}
\end{center}\vspace*{-3pt}

For all $I,J\in {\rm Id}(R)$, we have: $I/
_{\textstyle \sim _R}=J/
_{\textstyle \sim _R}$ iff $\sqrt{I}=\sqrt{J}$ iff $\iota\gamma _R(\sqrt{I})=\iota\gamma _R(\sqrt{J})$ iff $\iota\gamma _R(\bigcap \{P\in {\rm Spec}_{\rm Id}(R)\ |\ I\subseteq P\})=\iota\gamma _R(\bigcap \{Q\in {\rm Spec}_{\rm Id}(R)\ |\ J\subseteq Q\})$ iff $\bigcap \{\iota\gamma _R(P)\ |\ P\in {\rm Spec}_{\rm Id}(R),I\subseteq P\}=\bigcap \{\iota\gamma _R(Q)\ |\ Q\in {\rm Spec}_{\rm Id}(R),J\subseteq Q\}$ iff $\bigcap \{\iota\gamma _R(P)\ |\ P\in {\rm Spec}_{\rm Id}(R),\iota\gamma _R(I)\subseteq \iota\gamma _R(P)\}=\bigcap \{\iota\gamma _R(Q)\ |\ Q\in {\rm Spec}_{\rm Id}(R),\iota\gamma _R(J)\subseteq \iota\gamma _R(Q)\}$ iff $\bigcap \{\phi \ |\ \phi \in {\rm Spec}(R),\iota\gamma _R(I)\subseteq \phi \}=\bigcap \{\psi \ |\ \psi \in {\rm Spec}(R),\iota\gamma _R(J)\subseteq \psi \}$ iff $\rho _R(\iota\gamma _R(I))=\rho _R(\iota\gamma _R(J))$ iff $\iota\gamma _R(I)/
_{\textstyle \equiv _R}=\iota\gamma _R(J)/
_{\textstyle \equiv _R}$ iff $\varphi _R
(I/
_{\textstyle \sim _R})=\varphi _R
(J/
_{\textstyle \sim _R})$, hence $\varphi _R$ is well defined and injective. Since $\varphi _R\circ \nu _R=\lambda _R\circ \iota\gamma _R$, which is surjective, it follows that $\varphi _R$ is surjective. Clearly, $\varphi _R$ preserves the join and the meet, so it is a lattice isomorphism.

\noindent (\ref{hereshow2}) By (\ref{hereshow1}) and 
(\ref{hereshow0}), ${\rm Id}(R)/
_{\textstyle \sim _R}={\rm FGId}(R)/
_{\textstyle \sim _R}$, that is $\nu _R({\rm Id}(R))=\nu _R({\rm FGId}(R))$, so $\varphi _R(\nu _R({\rm Id}(R)))=\varphi _R(\nu _R({\rm FGId}(R)))$, which means that $\lambda _R(\iota\gamma _R({\rm Id}(R)))=\lambda _R(\iota\gamma _R({\rm FGId}(R)))$, thus $\lambda _R({\rm Con}(R))=\lambda _R({\cal K}(R))$, that is ${\rm Con}(R)/
_{\textstyle \equiv _R}={\cal K}(R)/
_{\textstyle \equiv _R}={\cal L}(R)$.\end{proof}

\begin{remark} Since ${\cal B}({\rm Con}(R))$ is a Boolean sublattice of ${\rm Con}(R)$ and ${\rm Con}(R)$ is isomorphic to ${\rm Id}(R)$, it follows that ${\cal B}({\rm Id}(R))$ is a Boolean sublattice of ${\rm Id}(R)$.\label{boolidcg}\end{remark}

\begin{corollary} $R^*$ and ${\cal L}(R)$ are frames and $\varphi _R:{\rm FGId}(R)/
_{\textstyle \sim _R}={\rm Id}(R)/
_{\textstyle \sim _R}=R^*\rightarrow {\rm Con}(R)/
_{\textstyle \equiv _R}={\cal K}(R)/
_{\textstyle \equiv _R}={\cal L}(R)$ is a lattice (and thus a frame) isomorphism, and, if we take, in the diagram above in the class of bounded lattices, the restrictions of the bounded lattice morphisms to the Boolean centers, then we get the following commutative diagram in the class of Boolean algebras, in which $\iota\gamma _R\mid _{{\cal B}({\rm Id}(R))}:{\cal B}({\rm Id}(R))\rightarrow {\cal B}({\rm Con}(R))$ and $\varphi _R\mid _{{\cal B}(R^*)}:{\cal B}(R^*)\rightarrow {\cal B}({\cal L}(R))$ are Boolean isomorphisms:\vspace*{-10pt}

\begin{center}
\begin{picture}(325,60)(0,0)
\put(63,40){${\cal B}({\rm Id}(R))$}
\put(70,5){${\cal B}(R^*)$}
\put(86,25){$\nu _R\mid _{{\cal B}({\rm Id}(R))}$}
\put(84,37){\vector(0,-1){22}}
\put(103,43){\vector(1,0){84}}
\put(213,38){\vector(0,-1){23}}
\put(215,25){$\lambda _R\mid _{{\cal B}({\rm Con}(R))}$}
\put(120,48){$\iota\gamma _R\mid _{{\cal B}({\rm Id}(R))}$}
\put(127,13){$\varphi _R\mid _{{\cal B}(R^*)}$}
\put(98,8){\vector(1,0){93}}
\put(189,40){${\cal B}({\rm Con}(R))$}
\put(193,5){${\cal B}({\cal L}(R))$}
\end{picture}
\end{center}\vspace*{-10pt}\label{izomreticulatii}\end{corollary}

\begin{proof} By Remark \ref{boolidcg}, Proposition \ref{hereshow} and the fact that ${\rm Con}(R)/
_{\textstyle \equiv _R}$ is a frame, along with $(1^{\circ })$.\end{proof}

So the bounded distributive lattices ${\cal L}(R)$ and $R^*$ are isomorphic, which was to be expected, since they have homeomorphic prime spectra of ideals w.r.t. the Stone topologies.

$R$ is called a {\em Baer ring} iff, for any $a\in R$, there exists an $e\in E(R)$ such that ${\rm Ann}_R(a)=eR$. By analogy to the case of bounded lattices, we shall call $R$ a {\em strongly Baer ring} iff, for any $U\subseteq R$, there exists an $e\in E(R)$ such that ${\rm Ann}_R(U)=eR$. Following \cite{bell,bell2}, we call $R$ a {\em semiprime ring} iff $\sqrt{\{0\}}=\{0\}$, that is $\displaystyle \bigcap _{P\in {\rm Spec}_{\rm Id}(R)}P=\{0\}$, which is equivalent to $\displaystyle \Delta _R=\iota\gamma _R(\{0\})=\iota\gamma _R(\bigcap _{P\in {\rm Spec}_{\rm Id}(R)}P)=\bigcap _{P\in {\rm Spec}_{\rm Id}(R)}\iota\gamma _R(P)=\bigcap _{\phi \in {\rm Spec}(R)}\phi =\rho _R(\Delta _R)$, which means that $R$ is a semiprime algebra.

\begin{remark}{\rm \cite{bell,bell2}} Let $I\in {\rm Id}(R)$. Then: $I/_{\textstyle \sim _R}=\top $ iff $\nu 
_R(I)=\nu 
_R(R)$ iff $\sqrt{I}=\sqrt{R}=R$ iff no prime ideal of $R$ includes $I$ iff $I=R$. If $R$ is semiprime, then we also have the following equivalences: $I/_{\textstyle \sim _R}=\perp $ iff $\nu 
_R(I)=\nu 
_R(\{0\})$ iff $I\subseteq \sqrt{I}=\sqrt{\{0\}}=\{0\}$ iff $I=\{0\}$.\label{rsemiprim}\end{remark}

\begin{remark} By Corollary \ref{izomreticulatii} and Remarks \ref{2.6} and \ref{2.7}, ${\cal A}nn({\rm Id}(R))\subseteq {\rm Id}({\rm Id}(R))$, and the canonical surjective lattice morphism $\nu _R:{\rm Id}(R)\rightarrow R^*={\rm Id}(R)/_{\textstyle \sim _R}$ preserves arbitrary intersections and joins and fulfills: $(I]_{{\rm Id}(R)}/_{\textstyle \sim _R}=(I/_{\textstyle \sim _R}]_{R^*}$ for any $I\in {\rm Id}(R)$.\label{idr}\end{remark}

\begin{corollary} For any nonzero cardinality $\kappa $ and any $h,i,j\in \overline{1,5}$, $(iv)_{R^*}$, $(h)_{\kappa ,R^*}$, $(i)_{<\infty ,R^*}$ and $(j)_{R^*}$ are equivalent.\label{condrstar}\end{corollary}

\begin{proof} By Corollaries \ref{cormydavey} and \ref{izomreticulatii}.\end{proof}

Let us consider the following conditions on $R$, where $\kappa $ is an arbitrary nonzero cardinality:

\begin{tabular}{ll}
$(1)_{\kappa ,R}$ & for each $U\subseteq R$ with $|U|\leq \kappa $, there exists an $e\in E(R)$ such that ${\rm Ann}_R(U)=eR$;\\ 
$(1)_{<\infty ,R}$ & for each finite $U\subseteq R$, there exists an $e\in E(R)$ such that ${\rm Ann}_R(U)=eR$;\\ 
$(1)_R$ & $R$ is a strongly Baer ring;\end{tabular}

\begin{tabular}{ll}
$(2)_{\kappa ,R}$ & $R$ is a Baer ring and $E(R)$ is a $\kappa $--complete Boolean algebra;\\ 
$(2)_{<\infty ,R}$ & $R$ is a Baer ring and $E(R)$ is a Boolean algebra;\\ 
$(2)_R$ & $R$ is a Baer ring and $E(R)$ is a complete Boolean algebra;\end{tabular}

\begin{tabular}{ll}
$(3)_{\kappa ,R}$ & ${\rm 2Ann}(R)$ is a $\kappa $--complete Boolean sublattice of ${\rm Id}(R)$ such that\\ 
& $I\mapsto {\rm Ann}_R({\rm Ann}_R(I))$ is a lattice morphism from ${\rm Id}(R)$ to ${\rm 2Ann}(R)$;\\ 
$(3)_{<\infty ,R}$ & ${\rm 2Ann}(R)$ is a Boolean sublattice of ${\rm Id}(R)$ such that\\ 
& $I\mapsto {\rm Ann}_R({\rm Ann}_R(I))$ is a lattice morphism from ${\rm Id}(R)$ to ${\rm 2Ann}(R)$;\\ 
$(3)_R$ & ${\rm 2Ann}(R)$ is a complete Boolean sublattice of ${\rm Id}(R)$ such that\\ 
& $I\mapsto {\rm Ann}_R({\rm Ann}_R(I))$ is a lattice morphism from ${\rm Id}(R)$ to ${\rm 2Ann}(R)$;;\end{tabular}

\begin{tabular}{ll}
$(4)_{\kappa ,R}$ & for all $I,J\in {\rm Id}(R)$, ${\rm Ann}_R(I\cap J)={\rm Ann}_R(I)\vee {\rm Ann}_R(J)$, and, for each $U\subseteq R$ with $|U|\leq \kappa $,\\ 
& there exists a finite subset $S\subseteq R$ such that ${\rm Ann}_R({\rm Ann}_R(U))={\rm Ann}_R(S)$;\\ 
$(4)_{<\infty ,R}$ & for all $I,J\in {\rm Id}(R)$, ${\rm Ann}_R(I\cap J)={\rm Ann}_R(I)\vee {\rm Ann}_R(J)$, and, for each finite $U\subseteq R$,\\ 
& there exists a finite subset $S\subseteq R$ such that ${\rm Ann}_R({\rm Ann}_R(U))={\rm Ann}_R(S)$;\\ 
$(4)_R$ & for all $I,J\in {\rm Id}(R)$, ${\rm Ann}_R(I\cap J)={\rm Ann}_R(I)\vee {\rm Ann}_R(J)$, and, for each $U\subseteq R$,\\ 
& there exists a finite subset $S\subseteq R$ such that ${\rm Ann}_R({\rm Ann}_R(U))={\rm Ann}_R(S)$;\\ 
$(iv)_R$ & for all $I,J\in {\rm Id}(R)$, ${\rm Ann}_R(I\cap J)={\rm Ann}_R(I)\vee {\rm Ann}_R(J)$;\end{tabular}

\begin{tabular}{ll}
$(5)_{\kappa ,R}$ & for each $U\subseteq R$ with $|U|\leq \kappa $, ${\rm Ann}_R(U)\vee {\rm Ann}_R({\rm Ann}_R(U))=R$;\\ 
$(5)_{<\infty ,R}$ & for each finite $U\subseteq R$, ${\rm Ann}_R(U)\vee {\rm Ann}_R({\rm Ann}_R(U))=R$;\\ 
$(5)_R$ & for each $U\subseteq R$, ${\rm Ann}_R(U)\vee {\rm Ann}_R({\rm Ann}_R(U))=R$.\end{tabular}

\begin{remark} Clearly, the properties from the second paragraph of Remark \ref{dualdavey} hold for $R$ instead of $L$, too. For the motivation behind the change of the second part of conditions $(4)_{\kappa ,L}$, $(4)_{<\infty ,L}$ and $(4)_L$ into what we see above in conditions $(4)_{\kappa ,R}$, $(4)_{<\infty ,R}$ and $(4)_R$, compare Lemma \ref{panncg} above to Lemma \ref{fannring} below.\end{remark}

Throughout the rest of this paper, $R$ shall be semiprime.

\begin{lemma}\begin{enumerate}
\item\label{rsirstar1}{\rm \cite[Lemma, p. 1863]{bell}} $\mu _R\mid _{E(R)}:E(R)\rightarrow {\cal B}(R^*)$ is a Boolean isomorphism.
\item\label{rsirstar2}{\rm \cite[Theorem $2.6$]{sim}} $R$ is a Baer ring iff $R^*$ is a Stone lattice.\end{enumerate}\label{rsirstar}\end{lemma}

\begin{proposition} For any nonzero cardinality $\kappa $, conditions $(2)_{\kappa ,R}$ and $(2)_{\kappa ,R^*}$ are equivalent.\label{2r}\end{proposition}

\begin{proof} By Lemma \ref{rsirstar}, (\ref{rsirstar1}) and (\ref{rsirstar2}).\end{proof}

\begin{lemma} $\nu _R\mid _{{\cal B}({\rm Id}(R))}:{\cal B}({\rm Id}(R))\rightarrow {\cal B}({\rm Id}(R))/_{\textstyle \sim _R}={\cal B}({\rm Id}(R)/_{\textstyle \sim _R})={\cal B}(R^*)$ is a Boolean isomorphism.\label{1or}\end{lemma}

\begin{proof} By $(1^{\circ })$ and Corollary \ref{izomreticulatii}.\end{proof}

\begin{lemma} For any $V,W\subseteq {\rm Id}(R)$, $I\in {\rm Id}(R)$ and $D\in {\cal B}({\rm Id}(R))$:\begin{enumerate}
\item\label{annidr2} ${\rm Ann}_{R^*}(I/_{\textstyle \sim _R})={\rm Ann}_{{\rm Id}(R)}(I)/_{\textstyle \sim _R}$ and ${\rm Ann}_{R^*}(W/_{\textstyle \sim _R})={\rm Ann}_{{\rm Id}(R)}(W)/_{\textstyle \sim _R}$;
\item\label{annidr6} ${\rm Ann}_{R^*}(V/_{\textstyle \sim _R})={\rm Ann}_{R^*}(W/_{\textstyle \sim _R})$ iff ${\rm Ann}_{{\rm Id}(R)}(V)/_{\textstyle \sim _R}={\rm Ann}_{{\rm Id}(R)}(W)/_{\textstyle \sim _R}$ iff ${\rm Ann}_{{\rm Id}(R)}(V)={\rm Ann}_{{\rm Id}(R)}(W)$;
\item\label{annidr7} $(D/_{\textstyle \sim _R}]_{R^*}={\rm Ann}_{R^*}(W/_{\textstyle \sim _R})$ iff $(D]_{{\rm Id}(R)}/_{\textstyle \sim _R}={\rm Ann}_{{\rm Id}(R)}(W)/_{\textstyle \sim _R}$ iff $(D]_{{\rm Id}(R)}={\rm Ann}_{{\rm Id}(R)}(W)$.\end{enumerate}\label{annidr}\end{lemma}

\begin{proof} (\ref{annidr2}) By Corollary \ref{izomreticulatii} and Lemma \ref{annlambda}.

\noindent (\ref{annidr6}) By (\ref{annidr2}), Corollary \ref{izomreticulatii} and Lemma \ref{annegal}, (\ref{annegal3}).

\noindent (\ref{annidr7}) By (\ref{annidr2}), Corollary \ref{izomreticulatii}, Lemma \ref{annegal}, (\ref{annegal3}), and Remark \ref{idr}.\end{proof}

\begin{lemma} For any $U\subseteq R$ and any $I\in {\rm Id}(R)$:\begin{itemize}
\item ${\rm Ann}_{{\rm Id}(R)}(I)=({\rm Ann}_R(I)]_{{\rm Id}(R)}$ and ${\rm Ann}_{{\rm Id}(R)}(\langle U\rangle _R)=({\rm Ann}_R(U)]_{{\rm Id}(R)}$;
\item ${\rm Ann}_{R^*}(\mu _R(I))={\rm Ann}_{R^*}(I/_{\textstyle \sim _R})=({\rm Ann}_R(I)]_{{\rm Id}(R)}/_{\textstyle \sim _R}=({\rm Ann}_R(I)/_{\textstyle \sim _R}]_{R^*}$;
\item ${\rm Ann}_{R^*}(\mu _R(U))={\rm Ann}_{R^*}(\mu _R(\langle U\rangle _R))=({\rm Ann}_R(U)]_{{\rm Id}(R)}/_{\textstyle \sim _R}=({\rm Ann}_R(U)/_{\textstyle \sim _R}]_{R^*}$;
\item ${\rm Ann}_{{\rm Id}(R)}({\rm Ann}_{{\rm Id}(R)}(I))=({\rm Ann}_R({\rm Ann}_R(I))]_{{\rm Id}(R)}$ and ${\rm Ann}_{R^*}({\rm Ann}_{R^*}(\mu _R(U)))=({\rm Ann}_R({\rm Ann}_R(U))/_{\textstyle \sim _R}]_{R^*}$;
\item ${\rm Ann}_{R^*}(I/_{\textstyle \sim _R})={\rm Ann}_{{\rm Id}(R)}(I)/_{\textstyle \sim _R}$ and ${\rm Ann}_{R^*}({\rm Ann}_{R^*}(I/_{\textstyle \sim _R}))={\rm Ann}_{{\rm Id}(R)}({\rm Ann}_{{\rm Id}(R)}(I))/_{\textstyle \sim _R}$.\end{itemize}\label{annid}\end{lemma}

\begin{proof} Let $J\in {\rm Id}(R)$. Then: $J\in ({\rm Ann}_R(I)]_{{\rm Id}(R)}$ iff $J\subseteq {\rm Ann}_R(I)$ iff $x\in {\rm Ann}_R(I)$ for all $x\in J$ iff $x\cdot y=0$ for all $x\in J$ and all $y\in I$ iff $J\cdot I=\{0\}$ iff $J\in ({\rm Ann}_R(I)]_{{\rm Id}(R)}$, hence ${\rm Ann}_{{\rm Id}(R)}(I)=({\rm Ann}_R(I)]_{{\rm Id}(R)}$, thus ${\rm Ann}_{{\rm Id}(R)}(\langle U\rangle _R)=({\rm Ann}_R(\langle U\rangle _R)]_{{\rm Id}(R)}=({\rm Ann}_R(U)]_{{\rm Id}(R)}$ by Lemma \ref{proprinel}.

By Lemma \ref{annidr}, (\ref{annidr2}), and Remark \ref{idr}, ${\rm Ann}_{R^*}(I/_{\textstyle \sim _R})=({\rm Ann}_R(I)]_{{\rm Id}(R)}/_{\textstyle \sim _R}=({\rm Ann}_R(I)/_{\textstyle \sim _R}]_{R^*}$. By Lemma \ref{proprinel} and Remark \ref{idr}, $\displaystyle {\rm Ann}_{R^*}(\mu _R(U))={\rm Ann}_{R^*}(\{\mu _R(x))\ |\ x\in U\})={\rm Ann}_{R^*}(\{\nu _R(xR))\ |\ x\in U\})=\bigcap _{x\in U}{\rm Ann}_{R^*}(\nu _R(xR))={\rm Ann}_{R^*}(\bigvee _{x\in U}\nu _R(xR))={\rm Ann}_{R^*}(\nu _R(\bigvee _{x\in U}xR))={\rm Ann}_{R^*}(\nu _R(\langle U\rangle _R))={\rm Ann}_{R^*}(\langle U\rangle _R/_{\textstyle \sim _R})$. Therefore ${\rm Ann}_{R^*}(\mu _R(I))={\rm Ann}_{R^*}(I/_{\textstyle \sim _R})$ and thus ${\rm Ann}_{R^*}(\mu _R(U))={\rm Ann}_{R^*}(\langle U\rangle _R/_{\textstyle \sim _R})={\rm Ann}_{R^*}(\mu _R(\langle U\rangle _R))$, hence, by the above, Lemma \ref{annidr}, (\ref{annidr2}), Remark \ref{idr} and Lemma \ref{proprinel}, ${\rm Ann}_{R^*}(\mu _R(U))={\rm Ann}_{R^*}(\langle U\rangle _R/_{\textstyle \sim _R})=({\rm Ann}_{{\rm Id}(R)}(\langle U\rangle _R)/_{\textstyle \sim _R}=({\rm Ann}_R(U)]_{{\rm Id}(R)}/_{\textstyle \sim _R}=({\rm Ann}_R(U)/_{\textstyle \sim _R}]_{R^*}$.

Hence ${\rm Ann}_{{\rm Id}(R)}({\rm Ann}_{{\rm Id}(R)}(I))={\rm Ann}_{{\rm Id}(R)}(({\rm Ann}_R(I)]_{{\rm Id}(R)})={\rm Ann}_{{\rm Id}(R)}({\rm Ann}_R(I))=({\rm Ann}_R({\rm Ann}_R(I))]_{{\rm Id}(R)}$ and ${\rm Ann}_{R^*}({\rm Ann}_{R^*}(\mu _R(U)))={\rm Ann}_{R^*}(({\rm Ann}_R(U)/_{\textstyle \sim _R}]_{R^*})={\rm Ann}_{R^*}({\rm Ann}_R(U)/_{\textstyle \sim _R})=({\rm Ann}_R({\rm Ann}_R(U))/_{\textstyle \sim _R}]_{R^*}$. By the above, ${\rm Ann}_{R^*}(I/_{\textstyle \sim _R})={\rm Ann}_{{\rm Id}(R)}(I)/_{\textstyle \sim _R}$, and thus ${\rm Ann}_{R^*}({\rm Ann}_{R^*}(I/_{\textstyle \sim _R}))={\rm Ann}_{R^*}({\rm Ann}_{{\rm Id}(R)}(I)/_{\textstyle \sim _R})={\rm Ann}_{{\rm Id}(R)}({\rm Ann}_{{\rm Id}(R)}(I))/_{\textstyle \sim _R}$.\end{proof}

\begin{lemma}For any $U,V\subseteq R$ and any $W\subseteq R^*$:\begin{enumerate}
\item\label{annfin1} ${\rm Ann}_{R^*}(\mu _R(U))={\rm Ann}_{R^*}(\mu _R(V))$ iff ${\rm Ann}_{R^*}(\langle U\rangle _R/_{\textstyle \sim _R})={\rm Ann}_{R^*}(\langle V\rangle _R/_{\textstyle \sim _R})$ iff ${\rm Ann}_R(U)/_{\textstyle \sim _R}={\rm Ann}_R(V)/$\linebreak $_{\textstyle \sim _R}$ iff ${\rm Ann}_R(U)={\rm Ann}_R(V)$;
\item\label{annfin2} there exists a finite subset $S\subseteq \langle U\rangle _R$ such that ${\rm Ann}_R(U)={\rm Ann}_R(S)$;
\item\label{annfin3} there exists a finite subset $S\subseteq R$ such that ${\rm Ann}_{R^*}(W)={\rm Ann}_{R^*}(\mu _R(S))$.\end{enumerate}\label{annfin}\end{lemma}

\begin{proof} (\ref{annfin1}) Let $U\subseteq R$. By Lemma \ref{annid} and Lemma \ref{annidr}, (\ref{annidr6}), ${\rm Ann}_{R^*}(\mu _R(U))={\rm Ann}_{R^*}(\mu _R(V))$ iff ${\rm Ann}_{R^*}(\langle U\rangle _R/$\linebreak $_{\textstyle \sim _R})={\rm Ann}_{R^*}(\langle V\rangle _R/_{\textstyle \sim _R})$ iff ${\rm Ann}_{{\rm Id}(R)}(\langle U\rangle _R)/_{\textstyle \sim _R}={\rm Ann}_{{\rm Id}(R)}(\langle V\rangle _R)/_{\textstyle \sim _R}$, which is equivalent both to $({\rm Ann}_R(U)/$\linebreak $_{\textstyle \sim _R}]_{R^*}=({\rm Ann}_R(V)/_{\textstyle \sim _R}]_{R^*}$ and to ${\rm Ann}_{{\rm Id}(R)}(\langle U\rangle _R)={\rm Ann}_{{\rm Id}(R)}(\langle V\rangle _R)$, which in turn are equivalent to ${\rm Ann}_R(U)/$\linebreak $_{\textstyle \sim _R}={\rm Ann}_R(V)/_{\textstyle \sim _R}$ and to $({\rm Ann}_R(U)]_{{\rm Id}(R)}=({\rm Ann}_R(V)]_{{\rm Id}(R)}$, respectively, the latter of which is equivalent to ${\rm Ann}_R(U)={\rm Ann}_R(V)$.

\noindent (\ref{annfin2}) By Proposition \ref{hereshow}, (\ref{hereshow1}), for an appropriate finite subset $S\subseteq \langle U\rangle _R$, $\langle U\rangle _R/_{\textstyle \sim _R}=\langle S\rangle _R/_{\textstyle \sim _R}$, thus ${\rm Ann}_{R^*}(\langle U\rangle _R/$\linebreak $_{\textstyle \sim _R})={\rm Ann}_{R^*}(\langle S\rangle _R/_{\textstyle \sim _R})$, hence ${\rm Ann}_R(U)={\rm Ann}_R(S)$ by (\ref{annfin1}).

\noindent (\ref{annfin3}) Let $W\subseteq R^*$, so that $W=\{I_k/_{\textstyle \sim _R}\ |\ k\in K\}$ for some $(I_k)_{k\in K}\subseteq {\rm Id}(R)$. As pointed out in Corollary \ref{izomreticulatii}, $R^*$ is a frame, thus $\displaystyle {\rm Ann}_{R^*}(W)={\rm Ann}_{R^*}(\bigvee _{k\in K}I_k/_{\textstyle \sim _R})={\rm Ann}_{R^*}((\bigvee _{k\in K}I_k)/_{\textstyle \sim _R})={\rm Ann}_{R^*}(\langle S\rangle _R/_{\textstyle \sim _R})={\rm Ann}_{R^*}(\mu _R(\langle S\rangle _R))={\rm Ann}_{R^*}(\mu _R(S))$ for some finite subset $\displaystyle S\subseteq \bigvee _{k\in K}I_k$, by Proposition \ref{hereshow}, (\ref{hereshow1}), and Lemma \ref{annid}.\end{proof}

\begin{lemma} ${\cal A}nn(R)=\{{\rm Ann}_R(S)\ |\ S\subseteq R,|S|<\infty \}$.\label{fannring}\end{lemma}

\begin{proof} By Lemma \ref{annfin}, (\ref{annfin2}).\end{proof}

\begin{proposition}\begin{enumerate}
\item\label{1r1} $(1)_{1,R}$ implies $(1)_{<\infty ,R}$;
\item\label{1r2} $(1)_{<\infty ,R}$ implies $(1)_{R^*}$;
\item\label{1r3} $(1)_{R^*}$ implies $(1)_R$;
\item\label{1r4} for any nonzero cardinality $\kappa $: $(1)_{\kappa ,R}$, $(1)_{<\infty ,R}$, $(1)_R$ and $(1)_{R^*}$ are equivalent.\end{enumerate}\label{1r}\end{proposition}

\begin{proof} (\ref{1r1}) If $n\in \N ^*$, $u_1,\ldots ,u_n\in R$ and, for each $i\in \overline{1,n}$, ${\rm Ann}_R(u_i)=e_iR$ for some $e_i\in E(R)$, then $\displaystyle {\rm Ann}_R(\{u_1,\ldots ,u_n\})=\bigcap _{i=1}^n{\rm Ann}_R(u_i)=\bigcap _{i=1}^ne_iR=e_1R\cdot \ldots \cdot e_nR=(e_1\cdot \ldots \cdot e_n)R$, and $e_1\cdot \ldots \cdot e_n=e_1\wedge \ldots \wedge e_n\in E(R)$.

\noindent (\ref{1r2}) Let $W\subseteq R^*$. By Lemma \ref{annfin}, (\ref{annfin3}), there exists a finite subset $S\subseteq R$ such that ${\rm Ann}_{R^*}(W)={\rm Ann}_{R^*}(\mu _R(S))$. If $(1)_{<\infty ,R}$ is fulfilled, then ${\rm Ann}_R(S)=eR$ for some $e\in E(R)$. By Lemma \ref{annid} and Lemma \ref{rsirstar}, (\ref{rsirstar1}), it follows that ${\rm Ann}_{R^*}(W)={\rm Ann}_{R^*}(\mu _R(S))=({\rm Ann}_R(S)/_{\textstyle \sim _R}]_{R^*}=(eR/_{\textstyle \sim _R}]_{R^*}=(\mu _R(e)]_{R^*}$, and $\mu _R(e)\in {\cal B}(R^*)$.

\noindent (\ref{1r3}) Let $U\subseteq R$, so that $\mu _R(U)\subseteq R^*$. If $(1)_{R^*}$ is fulfilled, then ${\rm Ann}_{R^*}(\mu _R(U))=(f]_{R^*}$ for some $f\in {\cal B}(R^*)$. By Lemma \ref{rsirstar}, (\ref{rsirstar1}), $f=\mu _R(e)=eR/_{\textstyle \sim _R}=\nu _R(eR)$ for some $e\in E(R)$, so that $eR\in {\cal B}({\rm Id}(R))$ by Lemma \ref{1or}. By Lemma \ref{annid} and Lemma \ref{annidr}, (\ref{annidr7}), ${\rm Ann}_{R^*}(\langle U\rangle _R/_{\textstyle \sim _R})={\rm Ann}_{R^*}(\mu _R(\langle U\rangle _R))=(eR/_{\textstyle \sim _R}]_{R^*}$, hence ${\rm Ann}_{{\rm Id}(R)}(\langle U\rangle _R)=(eR]_{{\rm Id}(R)}$, that is $({\rm Ann}_R(U)]_{{\rm Id}(R)}=(eR]_{{\rm Id}(R)}$, so that ${\rm Ann}_R(U)=eR$.

\noindent (\ref{1r4}) By (\ref{1r1}), (\ref{1r2}) and (\ref{1r3}).\end{proof}

\begin{remark} By Lemma \ref{proprinel}, ${\cal A}nn(R)=\{{\rm Ann}_R(I)\ |\ I\in {\rm Id}(R)\}$ and ${\rm 2Ann}(R)=\{{\rm Ann}_R({\rm Ann}_R(I))\ |\ I\in {\rm Id}(R)\}$.\label{annrgen}\end{remark}

Let us consider the following conditions on $R$:

\begin{tabular}{cl}
$(ann)_R$ & ${\cal A}nn(R)$ is a sublattice of ${\rm Id}(R)$ such that the map $I\mapsto {\rm Ann}_R(I)$ is\\
& a lattice anti--morphism from ${\rm Id}(R)$ to ${\cal A}nn(R)$;\\ 
$(2ann)_R$ & ${\rm 2Ann}(R)$ is a sublattice of ${\rm Id}(R)$ such that the map $I\mapsto {\rm Ann}_R({\rm Ann}_R(I))$ is\\
& a lattice morphism from ${\rm Id}(R)$ to ${\rm 2Ann}(R)$.\end{tabular}

\begin{lemma}\begin{enumerate}
\item\label{annrannidr1} The map ${\rm Ann}_R(I)\mapsto {\rm Ann}_{{\rm Id}(R)}(I)$ ($I\in {\rm Id}(R)$) from ${\cal A}nn(R)$ to ${\rm PAnn}({\rm Id}(R))$ is an order isomorphism.
\item\label{annrannidr2} The map ${\rm Ann}_R({\rm Ann}_R(I))\mapsto {\rm Ann}_{{\rm Id}(R)}({\rm Ann}_{{\rm Id}(R)}(I))$ ($I\in {\rm Id}(R)$) from ${\rm 2Ann}(R)$ to ${\rm P2Ann}({\rm Id}(R))$ is an order isomorphism.
\end{enumerate}\label{annrannidr}\end{lemma}

\begin{proof} Clearly, these maps are surjective and order--preserving. By Lemma \ref{proprinel}, they are completely defined. By Lemma \ref{annid}, the map from (\ref{annrannidr1}) is well defined and injective. By Lemma \ref{annid}, Lemma \ref{annidr}, (\ref{annidr2}), and the injectivity of the map from (\ref{annrannidr1}), it follows that the map from (\ref{annrannidr2}) is well defined and injective. Hence these maps are bijective. Clearly, their inverses are order--preserving, as well.\end{proof}

\begin{lemma}\begin{enumerate}
\item\label{annidrannr*1} The map ${\rm Ann}_{{\rm Id}(R)}(I)\mapsto {\rm Ann}_{{\rm Id}(R)}(I)/_{\textstyle \sim _R}={\rm Ann}_{R^*}(I/_{\textstyle \sim _R})$ ($I\in {\rm Id}(R)$) from ${\rm PAnn}({\rm Id}(R))$ to ${\rm PAnn}(R^*)$ is an order isomorphism.
\item\label{annidrannr*2} For all $I,J\in {\rm Id}(R)$: ${\rm Ann}_{{\rm Id}(R)}(I\cap J)={\rm Ann}_{{\rm Id}(R)}(I)\vee {\rm Ann}_{{\rm Id}(R)}(J)$ iff ${\rm Ann}_{R^*}(I/_{\textstyle \sim _R}\wedge J/_{\textstyle \sim _R})={\rm Ann}_{R^*}(I/_{\textstyle \sim _R})$\linebreak $\vee {\rm Ann}_{R^*}(J/_{\textstyle \sim _R})$.
\item\label{annidrannr*3} $(pann)_{{\rm Id}(R)}$ is equivalent to $(pann)_{R^*}$, and, if they are fulfilled, then the map from (\ref{annidrannr*1}) is a lattice isomorphism.\end{enumerate}\label{annidrannr*}\end{lemma}

\begin{proof} By Lemma \ref{annbij} and Corollary \ref{izomreticulatii}, with the equality in (\ref{annidrannr*1}) holding by Lemma \ref{annid}.\end{proof}

\begin{lemma}\begin{enumerate}
\item\label{annrannr*1} The map ${\rm Ann}_R(I)\mapsto {\rm Ann}_{R^*}(I/_{\textstyle \sim _R})$ ($I\in {\rm Id}(R)$) from ${\cal A}nn(R)$ to ${\rm PAnn}(R^*)$ is an order isomorphism.
\item\label{annrannr*2} For all $I,J\in {\rm Id}(R)$: ${\rm Ann}_R(I\cap J)={\rm Ann}_R(I)\vee {\rm Ann}_R(J)$ iff ${\rm Ann}_{R^*}(I/_{\textstyle \sim _R}\wedge J/_{\textstyle \sim _R})={\rm Ann}_{R^*}(I/_{\textstyle \sim _R})\vee {\rm Ann}_{R^*}(J/_{\textstyle \sim _R})$.
\item\label{annrannr*3} $(ann)_R$ is equivalent to $(pann)_{R^*}$ and, if they are fulfilled, then the map from (\ref{annrannr*1}) is a lattice isomorphism.\end{enumerate}\label{annrannr*}\end{lemma}

\begin{proof} By Remark \ref{annrgen}, Lemma \ref{annrannidr}, (\ref{annrannidr1}), Lemma \ref{annidrannr*} and the fact that the map from (\ref{annrannr*1}) is the composition of the map from Lemma \ref{annidrannr*}, (\ref{annidrannr*1}), with the map from Lemma \ref{annrannidr}, (\ref{annrannidr1}).\end{proof}

\begin{lemma}\begin{enumerate}
\item\label{2annidr2annr*1} The map ${\rm Ann}_{{\rm Id}(R)}({\rm Ann}_{{\rm Id}(R)}(I))\mapsto {\rm Ann}_{{\rm Id}(R)}({\rm Ann}_{{\rm Id}(R)}(I))/_{\textstyle \sim _R}={\rm Ann}_{R^*}({\rm Ann}_{R^*}(I/_{\textstyle \sim _R}))$ ($I\in {\rm Id}(R)$) from ${\rm P2Ann}({\rm Id}(R))$ to ${\rm P2Ann}(R^*)$ is an order isomorphism.
\item\label{2annidr2annr*2} $(p2ann)_{{\rm Id}(R)}$ is equivalent to $(p2ann)_{R^*}$ and, if these conditions are fulfilled, then the map from (\ref{echiv(iii)nobm1}) is a lattice isomorphism.\end{enumerate}\label{2annidr2annr*}\end{lemma}

\begin{proof} By Lemma \ref{echiv(iii)nobm} and Corollary \ref{izomreticulatii}, with the equality in (\ref{2annidr2annr*1}) holding by Lemma \ref{annid}.\end{proof}

\begin{lemma}\begin{enumerate}
\item\label{2annr2annr*1} The map ${\rm Ann}_R({\rm Ann}_R(I))\mapsto {\rm Ann}_{R^*}({\rm Ann}_{R^*}(I/_{\textstyle \sim _R}))$ ($I\in {\rm Id}(R)$) from ${\rm 2Ann}(R)$ to\linebreak ${\rm P2Ann}(R^*)$ is an order isomorphism.
\item\label{2annr2annr*2} $(2ann)_R$ is equivalent to $(p2ann)_{R^*}$ and, if these conditions are fulfilled, then the map from (\ref{2annr2annr*1}) is a lattice isomorphism.\end{enumerate}\label{2annr2annr*}\end{lemma}

\begin{proof} By Remark \ref{annrgen}, Lemma \ref{annrannidr}, (\ref{annrannidr2}), Lemma \ref{2annidr2annr*} and the fact that the map from (\ref{2annr2annr*1}) is the composition of the map from Lemma \ref{2annidr2annr*}, (\ref{2annidr2annr*1}), with the map from Lemma \ref{annrannidr}, (\ref{annrannidr2}).\end{proof}

\begin{proposition} For any nonzero cardinality $\kappa $, the properties $(3)_{\kappa ,R}$ and $(3)_{\kappa ,R^*}$ are equivalent.\label{3r}\end{proposition}

\begin{proof} By Lemma \ref{2annr2annr*} and the fact that the map from Lemma \ref{2annr2annr*}, (\ref{2annr2annr*1}), composed with the map from condition $(3)_{\kappa ,R}$ equals the map from condition $(3)_{\kappa ,R^*}$ composed with the canonical surjective lattice morphism from ${\rm Id}(R)$ to ${\rm Id}(R)/_{\textstyle \sim _R}=R^*$.\end{proof}

\begin{proposition} For any nonzero cardinality $\kappa $, conditions $(4)_{\kappa ,R}$, $(4)_{<\infty ,R}$, $(4)_R$, $(iv)_R$ and $(iv)_{R^*}$ are equivalent.\label{4r}\end{proposition}

\begin{proof} By Lemma \ref{annrannr*}, (\ref{annrannr*2}), $(iv)_R$ is equivalent to $(iv)_{R^*}$. By Lemma \ref{fannring}, $\{{\rm Ann}_R(S)\ |\ S\subseteq R,|S|<\infty \}={\cal A}nn(R)\supseteq {\rm 2Ann}(R)$, which means that, for any nonzero cardinality $\kappa $, the second property in $(4)_{\kappa ,R}$ is fulfilled, so that conditions $(4)_{\kappa ,R}$, $(4)_{<\infty ,R}$, $(4)_R$ and $(iv)_R$ are equivalent.\end{proof}

\begin{lemma}\begin{enumerate}
\item\label{bidrann1} ${\cal B}({\rm Id}(R))=\{eR\ |\ e\in E(R)\}$.
\item\label{bidrann2} If $U\subseteq R$, then $U\cap {\rm Ann}_R(U)\subseteq \{0\}$. If $I\in {\rm Id}(R)$, then $I\cap {\rm Ann}_R(I)=\{0\}$.
\item\label{bidrann4} If $U\subseteq R$ such that ${\rm Ann}_R(U)\vee {\rm Ann}_R({\rm Ann}_R(U))=R$, then ${\rm Ann}_R(U)=eR$ for some $e\in E(R)$.\end{enumerate}\label{bidrann}\end{lemma}

\begin{proof} (\ref{bidrann1}) This statement may be known, but, for the sake of completeness, I am deriving it from the statements in this paper. Let $\pi _R=(\nu _R\mid _{{\cal B}({\rm Id}(R))}^{-1})\circ \mu _R\mid _{E(R)}:E(R)\rightarrow {\cal B}({\rm Id}(R))$ be the composition of the inverse of the Boolean isomorphism from Corollary \ref{1or} to the Boolean isomorphism from Lemma \ref{rsirstar}, (\ref{rsirstar1}). For any $e\in E(R)$, since $\mu _R(e)=eR/_{\textstyle \sim _R}=\nu _R(eR)$, it follows that $\pi _R(e)=eR$. Hence ${\cal B}({\rm Id}(R))=\pi _R(E(R))=\{eR\ |\ e\in E(R)\}$.

\noindent (\ref{bidrann2}) If $U\subseteq R$ and $x\in U\cap {\rm Ann}_R(U)$, then $x\cdot x=0$, so that $x=0$ since a semiprime commutative unitary ring has no nonzero nilpotents \cite[p.125,126]{kist}. Now, if $I\in {\rm Id}(R)$, then $0\in I\cap {\rm Ann}_R(I)$.

\noindent (\ref{bidrann4}) By (\ref{bidrann2}), it follows that ${\rm Ann}_R(U)\in {\cal B}({\rm Id}(R))$, having ${\rm Ann}_R({\rm Ann}_R(U))$ as a complement, so that ${\rm Ann}_R(U)=eR$ for some $e\in E(R)$ by (\ref{bidrann1}).\end{proof}

\begin{proposition} For any nonzero cardinality $\kappa $:\begin{enumerate}
\item\label{5r1} $(5)_{\kappa ,R}$ implies $(1)_{\kappa ,R}$;
\item\label{5r2} $(5)_{<\infty ,R}$ implies $(5)_{R^*}$;
\item\label{5r3} $(5)_{R^*}$ implies $(5)_R$;
\item\label{5r4} $(5)_{\kappa ,R}$, $(5)_{<\infty ,R}$, $(5)_R$ and $(5)_{R^*}$ are equivalent.\end{enumerate}\label{5r}\end{proposition}

\begin{proof} (\ref{5r1}) By Lemma \ref{bidrann}, (\ref{bidrann4}).

\noindent (\ref{5r2}) Let $W\subseteq R^*$, so that, by Lemma \ref{annfin}, (\ref{annfin3}), ${\rm Ann}_{R^*}(W)={\rm Ann}_{R^*}(\mu _R(S))$ for some finite subset $S\subseteq R$. If $(5)_{<\infty ,R}$ is fulfilled, then ${\rm Ann}_R(S)\vee {\rm Ann}_R({\rm Ann}_R(S))=R$. By Lemma \ref{annid}, it follows that ${\rm Ann}_{R^*}(W)\vee {\rm Ann}_{R^*}({\rm Ann}_{R^*}(W))={\rm Ann}_{R^*}(\mu _R(S))\vee {\rm Ann}_{R^*}({\rm Ann}_{R^*}(\mu _R(S)))=({\rm Ann}_R(S)/_{\textstyle \sim _R}]_{R^*}\vee ({\rm Ann}_R({\rm Ann}_R(S))/_{\textstyle \sim _R}]_{R^*}=(({\rm Ann}_R(S)\vee {\rm Ann}_R({\rm Ann}_R(S)))/_{\textstyle \sim _R}]_{R^*}=(R/_{\textstyle \sim _R}]_{R^*}=(\top ]_{R^*}=R^*$.

\noindent (\ref{5r3}) Let $U\subseteq R$, so that $\mu _R(S)\subseteq R^*$, thus, if $(5)_{R^*}$ is fulfilled, then, by Lemma \ref{annid}, $R^*={\rm Ann}_{R^*}(\mu _R(U))\vee {\rm Ann}_{R^*}({\rm Ann}_{R^*}(\mu _R(U)))\!=\!({\rm Ann}_R(U)/_{\textstyle \sim _R}]_{R^*}\vee ({\rm Ann}_R({\rm Ann}_R(U))/_{\textstyle \sim _R}]_{R^*}\!=\!(({\rm Ann}_R(U)\vee {\rm Ann}_R({\rm Ann}_R(U)))/_{\textstyle \sim _R}]_{R^*}$, so that $({\rm Ann}_R(U)\vee {\rm Ann}_R({\rm Ann}_R(U)))/_{\textstyle \sim _R}=\top $, therefore ${\rm Ann}_R(U)\vee {\rm Ann}_R({\rm Ann}_R(U))=R$ by Remark \ref{rsemiprim}.

\noindent (\ref{5r4}) By (\ref{5r1}), (\ref{5r2}), (\ref{5r3}), Proposition \ref{1r} and Corollary \ref{condrstar}.\end{proof}

\begin{theorem} For any nonzero cardinality $\kappa $ and any $h,i,j\in \overline{1,5}$:\begin{enumerate}
\item\label{myringsdavey1} $(iv)_R$, $(h)_{\kappa ,R}$, $(i)_{<\infty ,R}$ and $(j)_R$ are equivalent;
\item\label{myringsdavey2} $(iv)_{{\rm Id}(R)}$, $(h)_{\kappa ,{\rm Id}(R)}$, $(i)_{<\infty ,{\rm Id}(R)}$ and $(j)_{{\rm Id}(R)}$ are equivalent;
\item\label{myringsdavey3} the conditions from (\ref{myringsdavey1}) are equivalent to those from (\ref{myringsdavey2}) and to the equivalent conditions $(iv)_{R^*}$, $(h)_{\kappa ,R^*}$, $(i)_{<\infty ,R^*}$ and $(j)_{R^*}$.\end{enumerate}

In particular, $R$ is a Baer ring iff $R$ is a strongly Baer ring iff ${\rm Id}(R)$ is a Stone lattice iff ${\rm Id}(R)$ is a strongly Stone lattice iff $R^*$ is a Stone lattice iff $R^*$ is a strongly Stone lattice.\label{myringsdavey}\end{theorem}

\begin{proof} By Corollaries \ref{izomreticulatii} and \ref{condrstar} and Propositions \ref{1r}, \ref{2r}, \ref{3r}, \ref{4r} and \ref{5r}.\end{proof}

Throughout the rest of this paper, $S$ shall be a semiprime commutative unitary ring.

\begin{remark} By Corollary \ref{izomreticulatii} and Remark \ref{prodalg}, $R\times S$ is semiprime, ${\rm Id}(R\times S)\cong {\rm Id}(R)\times {\rm Id}(S)$ and $(R\times S)^*\cong R^*\times S^*$.\end{remark}

\begin{corollary} For any nonzero cardinality $\kappa $ and any $h,i,j\in \overline{1,5}$: the equivalent conditions $(iv)_{R\times S}$, $(h)_{\kappa ,R\times S}$, $(i)_{<\infty ,R\times S}$ and $(j)_{R\times S}$ are fulfilled iff the equivalent conditions $(iv)_R$, $(h)_{\kappa ,R}$, $(i)_{<\infty ,R}$ and $(j)_R$, as well as the equivalent conditions $(iv)_S$, $(h)_{\kappa ,S}$, $(i)_{<\infty ,S}$ and $(j)_S$, are fulfilled.\end{corollary}

Concerning the semiprimality condition, recall that it does not need to be enforced for residuated lattices or bounded distributive lattices because they are congruence--distributive and thus semiprime.

\section{Conclusions}
\label{conclusions}

It may be possible to extend Theorem \ref{mydavey} to other kinds of congruence lattices, possibly of algebras from non--modular commutator varieties. If the congruence lattices of the algebras from such varieties cover entire varieties of bounded lattices, then all lattices from those varieties fulfill \cite[Theorem $1$]{dav}. The study of such extensions of Davey`s Theorem remains a theme for future research.

Another important research theme is finding more classes of algebras in which, given an appropriate setting (regarding definitions for annihilators and a Boolean center), Davey`s Theorem holds not only for congruences, but also for elements, as in the case of bounded distributive lattices, residuated lattices and commutative unitary rings.

\end{document}